\documentclass[pdflatex,sn-mathphys-num]{sn-jnl}

\usepackage{graphicx}%
\usepackage{multirow}%
\usepackage{amsmath,amssymb,amsfonts}%
\usepackage{amsthm}%
\usepackage{mathrsfs}%
\usepackage[title]{appendix}%
\usepackage{xcolor}%
\usepackage{textcomp}%
\usepackage{manyfoot}%
\usepackage{booktabs}%
\usepackage{algorithm}%
\usepackage{algorithmic}%
\usepackage{listings}%
\usepackage{thmtools}%

\theoremstyle{thmstyleone}
\newtheorem{theorem}{Theorem}  
\newtheorem{proposition}{Proposition}  
\newtheorem{definition}{Definition}
\newtheorem{lemma}{Lemma}

\theoremstyle{thmstyletwo}

\theoremstyle{thmstylethree}

\begin{document}

\title[Zeros of $\Phi(a,b,z)$ With Variable $a$]{Bounding the Gap Between Zeros of the Variable-\\Parameter Confluent Hypergeometric Function}

\author{\fnm{Steven} \sur{Langel}}\email{slangel@mitre.org}

\affil{\orgname{The MITRE Corporation}, \orgaddress{\city{Bedford}, \postcode{01730}, \state{MA}, \country{USA}}}

\abstract{This paper derives a lower bound on the spacing between adjacent zeros of the confluent hypergeometric function $\Phi(a,b,z)$ when $a$ is variable and $(b,z) \in \mathbb{R}^+$ are known and fixed. Monotonicity of the bound is established, and the results are used to assess the accuracy of asymptotic approximations for the first passage probability of a Wiener process.\\ \\Approved for public release; distribution unlimited. Public release case number 25-2921. \\ \\ NOTICE: This technical data was produced for the U. S. Government under Contract No. FA870225CB001, and is subject to the Rights in Technical Data-Noncommercial Items Clause DFARS 252.227-7013 (FEB 2014).}




\maketitle

\insert\footins{\noindent\footnotesize\textcopyright\ 2026 The MITRE Corporation. ALL RIGHTS RESERVED.}

\section{Introduction}

The confluent hypergeometric function $\Phi(a,b,z)$ arises in the solution to many problems in science and engineering. It is particularly relevant in quantum mechanics, where it is the solution to Schr\"odinger's equation for a variety of potentials, including the Coulomb, harmonic and Morse potentials \cite{Ishkhanyan}. Other applications include optics, quantum chemistry, classical electrodynamics, heat transfer and general relativity (see \cite{Mathews} and the references therein). Most often, $z$ is the variable with $a$ and $b$ representing known physical parameters. However, there are exceptions. In Coulomb scattering, for example, expansions of $\Phi(a,b,z)$ in powers of $a$ have been used to gain deeper insights into Born approximations for the scattering wave function \cite{Gasaneo}. Another instance occurs in the study of first passage phenomena, where the goal is to determine the probability that a random event first happens at some time $t$. This problem is applicable to many topics, including Brownian motion, cellular mutation, development of optimal financial strategies, the formation of dark matter halos, and fault detection in communication systems \cite{Redner}, \cite{Masoliver}.

Consider the first passage problem for a scalar Ornstein-Uhlenbeck process $x(t)$, where we are interested in determining the probability $w(\tau)$ that $\lvert x(t) \rvert$ first crosses a threshold $c$ at some time $t \geq \tau$. Assuming that the initial value $x(0) = x$ is such that $\lvert x \rvert < c$, it can be shown that the Laplace transform of $w(\tau)$ is \cite[Eq. (7)]{Dirkse}

\begin{equation*}
	\mathcal{L}\{w(\tau)\} = \frac{c e^{-c^2/2}}{\sqrt{2\pi}} \frac{\Phi(s+1, 3/2, c^2/2)}{s \Phi(s,1/2,c^2/2)}
\end{equation*}\\
\noindent Notice that the Laplace variable $s$ appears in the first parameter of the hypergeometric functions. Since an analytic expression does not exist for the inverse transform, an approximate inversion is derived in \cite{Dirkse} for $c \rightarrow \infty$ using asymptotic expansions of the hypergeometric functions. To the best of our knowledge, there has been no rigorous investigation into how accurate such approximations are for the first passage probability. The results of this paper provide the ability to perform such an assessment.

To see how, first note that given the zeros of $\Phi(s, 1/2, c^2/2)$, all of which are real and simple \cite[pp. 185--186]{Buchholz}, $w(\tau)$ is expressible as a residue expansion\footnote[1]{This claim is not obvious. Reference \cite{Ricciardi} alludes to its validity, but does not provide a proof. We prove in Appendix \ref{Appendix B} that the inverse of $\mathcal{L}\{(w(\tau))\}$ can indeed be written as a residue expansion.}. The zeros and corresponding residues can be computed with high precision using numerical methods. However, not \emph{all} of the residues can be obtained because there are an infinite number of zeros \cite[p. 185]{Buchholz}. Thus, at best $w(\tau)$ can be written as a known, finite sum of residues plus some unknown truncation error. We will see in Section \ref{sec:first passage problem} that lower and upper bounds on the truncation error are obtainable if a lower bound on the spacing between adjacent zeros of $\Phi(s,1/2,c^2/2)$ can be found. This allows $w(\tau)$ to be placed within a known interval which subsequently allows one to assess the accuracy of existing approximations for $w(\tau)$.

Little work has been done concerning the distribution of zeros of $\Phi(a,b,z)$ when $a$ is the variable and $(b,z)$ are fixed. We already mentioned some properties, namely, that for $(b,z)\in \mathbb{R}^+$, the zeros are real and simple, and occur in infinite sets. Another important property is that each zero $a^*$ increases as $z$ increases \cite[p. 187]{Buchholz}, a fact that we will use to establish a link between $a^*$ and the zeros $z^*$ of $\Phi(a,b,z)$ when $(a,b)$ are fixed. This connection is crucial because there are numerous results concerning the distribution of $z^*$ that can be leveraged to gain insight into the distribution of $a^*$. One result relevant to this work is the lower bound in \cite[Eq. (83)]{Deano} on the ratio of two consecutive, positive real zeros of $\Phi(a,b,z)$ when $a$ and $b$ take on real, fixed values. We will see that \cite[Eq. (83)]{Deano} is the key to obtaining a lower bound on the gap between consecutive zeros of $\Phi(a,b,z)$ for $(b,z) \in \mathbb{R}^+$.

The paper is organized as follows. A summary of the two main theorems proved in this work and some preliminary results are given in Section \ref{sec:main results}. In Section \ref{sec:properties of a_star}, we show that the spacing $\Delta a$ between consecutive zeros of $\Phi(a,b,z)$ for known $(b,z) \in \mathbb{R}^+$ is governed by the solution to an initial value problem (IVP). We use a comparison theorem in Section \ref{sec:analytic bound} to approximate the IVP so that an analytic lower bound on $\Delta a$ is obtainable. The bound, which we prove is monotonic in Section \ref{sec:monotonicity}, is subsequently used in Section \ref{sec:first passage problem} to analyze the accuracy of asymptotic approximations for the first passage probability of a Wiener process. Conclusions and recommendations for future work are given in Section \ref{sec:conclusions}.

\section{Main Contributions and Preliminary Results}\label{sec:main results}

This paper will prove the following two Theorems.
\begin{restatable}{theorem}{FirstTheorem}\label{primary theorem}
	Let $\Phi(a, b, z_l)$ be the confluent hypergeometric function of the first kind, where $(b, z_l) \in \mathbb{R}^+$ are known and fixed, and let $a_k^* < a_{k-1}^*$ be two consecutive real zeros of $\Phi(a,b,z_l)$. Then with $g_k = e^{2\pi / \sqrt{(b-2a_k^*)^2 - (b-1)^2}}$ and $\beta_k = b - a_k^* - 1$, if $z_l < \beta_{k-1} / g_{k-1}$, a lower bound on $\Delta a = a_{k-1}^* - a_k^*$ is
	
	\begin{equation*}
		\Delta a \geq \beta_k - \frac{\beta_k}{4 g_k} \left[ 2 + \sqrt{z_l/\beta_k} (g_k - 1) \right]^2
	\end{equation*}
\end{restatable}

\begin{restatable}{theorem}{SecondTheorem}\label{secondary theorem}
	For $b\in \mathbb{R}^+$, let $\{y\}$ be the set of roots of the polynomial
	
	\begin{equation*}
		\frac{1}{4\pi^2} y^6 - \frac{1}{\pi} y^5 + y^4 - \frac{(b-1)^2}{\pi} y^3 + (b^2 - 2) y^2 + (b-1)^2 (2b - 3) = 0
	\end{equation*}\\
	For the $i$th root $y_i$, let $\bar{a}_i = (b/2) - (1/2)[(b-1)^2 + y_i^2]^{1/2}$. Then with
	
	\begin{equation*}
		\begin{aligned}
			\bar{a}^* = &\min_{1\leq i \leq 6} \bar{a}_i \\[1ex]
			&\mathrm{such} \; \mathrm{that} \; \frac{4\pi (-1 - \bar{a}_i + b) (b - 2 \bar{a}_i)}{[(b - 2 \bar{a}_i) ^ 2 - (b - 1) ^ 2]^{3/2}} = 1 \; \mathrm{and} \; \mathrm{Im} (\bar{a}_i) = 0
		\end{aligned}
	\end{equation*}\\
	\noindent the bound $\Delta a$ is a monotonically decreasing function of $a_k^*$ for $a_k^* \in (-\infty, \bar{a}^*)$.
\end{restatable}

\noindent Some preliminary definitions and results are provided first that will lay the foundation for the technical developments of the paper.

\begin{definition}\label{def:confluent hypergeometric function}
	The confluent hypergeometric function is defined by the power series
	
	\begin{equation}
		\label{eq:confluent hypergeometric function}
		\Phi(a, b, z) = \sum_{n=0}^{\infty} \frac{(a)_n}{(b)_n n!}z^n
	\end{equation}\\
	\noindent where $(a)_n = a(a+1)\cdots(a+n-1)$. It is well known \cite{Hazewinkel} that as either a function of $z$ with $a$ and $b$ fixed, or as a function of $a$ with $b$ and $z$ fixed, $\Phi(a,b,z)$ is an entire function. It is a meromorphic function of $b$ with $a$ and $z$ fixed with simple poles at $b = 0, -1, -2, \dots$.
\end{definition}

\begin{definition}\label{def:Whittaker M function}
	The function $\Phi(a,b,z)$ is related to Whittaker's $\mathcal{M}$ function, defined as \cite{Buchholz}\footnote[2]{The normalizing factor $\Gamma(1+\mu)$ ensures that $\mathcal{M}_{\varkappa, \mu/2}(z)$ is defined even when $\mu$ is a negative integer.}
	
	\begin{equation}
		\label{eq:Whittaker M function}
		\mathcal{M}_{\varkappa, \mu/2}(z) = \frac{1}{\Gamma(1+\mu)} z^{(1+\mu)/2}e^{-z/2} \Phi\left(\frac{1+\mu}{2} - \varkappa, 1 + \mu, z\right)
	\end{equation}
\end{definition}

\begin{proposition}\label{prop:definite integrals}
	With $\text{Re}(b) > 0$ and $\xi \neq \eta$,
	
	\begin{equation}
		\label{eq:integral 1}
		\begin{split}
			\int\limits_0^z t^{b-1}e^{-t} & \Phi(\xi,b,t)\Phi(\eta,b,t)dt \\ & = \frac{e^{-z}z^b}{b(\eta - \xi)} [\eta \Phi(\xi,b,z)\Phi(\eta + 1, b + 1, z) - \xi \Phi(\eta, b, z) \Phi(\xi + 1, b + 1, z)]
		\end{split}
	\end{equation}\\
	
	\noindent Given $\Phi(a,b,t)$ with $\text{Re}(b)>0$ and $k=b/2-a$,
	
	\begin{equation}
		\label{eq:integral 2}
		\begin{split}
			\int\limits_0^z \left( \frac{k}{t} - \frac{1}{2} \right) & e^{-t} t^b \Phi^2(a,b,t) dt \\ & = z^{b+1} e^{-z} \left[ \left( \frac{2k - b + 1}{2z} \right) \Phi^2(a,b,z) + \left(\frac{a}{b}\right)^2 \Phi^2(a+1,b+1,z) \right] \\ & + z^b e^{-z} \frac{a}{b} (b-z-1) \Phi(a,b,z) \Phi(a+1,b+1,z)
		\end{split}
	\end{equation}
\end{proposition}

\begin{proof}
	See Appendix \ref{Appendix A}.
\end{proof}

\begin{proposition}\label{prop:zeros a_star of Phi}
	For $(b,z) \in \mathbb{R}^+$ known and fixed, the zeros $a^*$ of $\Phi(a,b,z)$ are real and simple, and all reside on the $-a$ axis\footnote[3]{The zeros $a^*$ must also occur in infinite sets with $a^* = -\infty$ as a limiting point (see \cite[p. 185]{Buchholz}).}.
\end{proposition}

\begin{proof}
	Reference \cite{Buchholz} proves that for real $\mu$ and $z$, with $\mu>-1$ and $z>0$, the zeros $\varkappa^*$ of $\mathcal{M}_{\varkappa,\mu/2}(z)$ are real and simple. Given the definition in (\ref{eq:Whittaker M function}), these zeros must also be the zeros of $\Phi((1+\mu)/2 - \varkappa,1+\mu,z)$, and from the series definition in (\ref{eq:confluent hypergeometric function}), will only occur when $\varkappa^* > (1+\mu)/2$. Otherwise, every term in the series will be positive given that $\mu>-1$ and $z>0$. In terms of $a = (1+\mu)/2 - \varkappa$ and $b = 1 + \mu$ we can thus conclude that the zeros $a^*$ of $\Phi(a,b,z)$ for $(b,z)\in \mathbb{R}^+$ are real and simple, and must reside on the $-a$ axis.
\end{proof}

\begin{proposition}\label{prop:number of zeros z_star}
	Let $-\infty < a < 0$ and $b>0$ be known and fixed. Then the number $N$ of positive real zeros $z^*$ of $\Phi(a,b,z)$ is given by

	\begin{equation}
		N = -[a]
	\end{equation}\\
	\noindent such that $[a]$ is the largest integer less than or equal to $a$.
\end{proposition}

\begin{proof}
	Equation (8$\alpha$) in \cite[p. 182]{Buchholz} states that with $+\infty > \varkappa \geq (1+\mu)/2$ and $\mu > -1$, the number $N$ of positive real zeros $z^*$ of $z^{-(1+\mu)/2} \mathcal{M}_{\varkappa,\mu/2}(z)$ is
	
	\begin{equation}
		\label{eq:number of zeros}
		N = -\left[\frac{1+\mu}{2} - \varkappa\right]
	\end{equation}\\
	\noindent Since $e^{-z/2}$ is an entire function, the zeros $z^*$ of $z^{-(1+\mu)/2} \mathcal{M}_{\varkappa,\mu/2}(z) = e^{-z/2} \Phi((1+\mu)/2 - \varkappa,1+\mu,z)$ are also the zeros of $\Phi((1+\mu)/2-\varkappa, 1+\mu,z)$. Then after making the substitutions $\varkappa=b/2-a$ and $\mu=b-1$ in (\ref{eq:number of zeros}), the result follows.
\end{proof}

The zero sequences $a^*$ and $z^*$ described in Propositions \ref{prop:zeros a_star of Phi} and \ref{prop:number of zeros z_star} will both be needed to prove Theorems \ref{primary theorem} and \ref{secondary theorem}. Therefore, it is instructive to have a labeling scheme for the elements of each set. We will adopt the scheme in Fig. \ref{fig:labeling scheme}.
\begin{figure}[!htb]
	\centering
	\includegraphics[width=\textwidth]{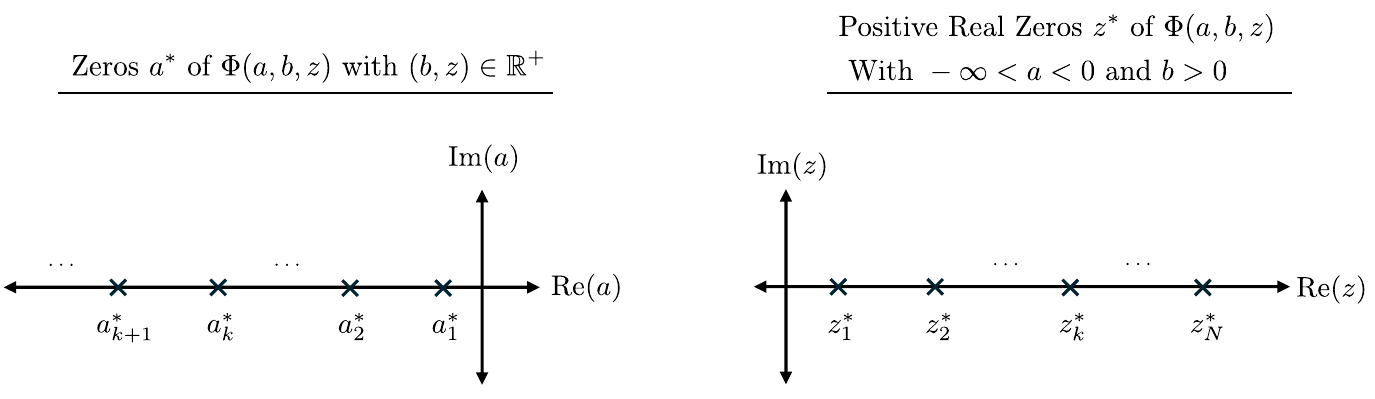}
	\caption{Labeling scheme for zero sequences $a^*$ and $z^*$.}
	\label{fig:labeling scheme}
\end{figure}

\begin{proposition}\label{prop:z-domain zero spacing}
	Consider real parameters $a$ and $b$ with $a<0$ and $b-a-1>0$, and let $0<z_1^* < z_2^* < \cdots < z_N^*$ be the positive real zeros of $\Phi(a,b,z)$. Then for any two consecutive zeros $z_l^*$ and $z_{l+1}^*$, the following inequality holds
	
	\begin{equation}
		\frac{z_{l+1}^*}{z_l^*} > \exp \left( \frac{2\pi}{\sqrt{(b-2a)^2 - (b-1)^2}} \right) \, , \, l \leq N-1
	\end{equation}
\end{proposition}

\begin{proof}
	See \cite[Eq. (83)]{Deano}.
\end{proof}

\section{Properties of $a^*$}\label{sec:properties of a_star}

For $(b,z) \in \mathbb{R}^+$ known and fixed, all zeros $a^*$ of $\Phi(a,b,z)$ move closer to the origin as $z$ increases. To prove this statement, we first make an observation concerning the local behavior of $a^*$ as $z$ undergoes small variations. Consider the pair $(a^*,z^*)$, which is a solution to $\Phi(a^*,b,z^*)=0$. Given that $\Phi(a,b,z)$ is an entire function in $a$ and $z$, it is continuously differentiable everywhere. Furthermore, $\partial \Phi / \partial a \neq 0$ for any pair $(a^*,z^*)$ because $a^*$ is a simple zero. By the implicit function theorem \cite{Thomson}, a unique, differentiable function $\varphi$ exists such that $a^* = \varphi(z^*)$ and $\Phi(\varphi(z), b, z)=0$ for all $z$ in some open interval containing $z^*$. Thus, small changes in $z$ are accompanied by small changes in $a^*$. With this in mind, consider the following result, with $\Phi(a^*,b,z)=0$ \cite[p. 113, Eq. (4$\alpha$)]{Buchholz}

\begin{equation}
	\label{eq:4 alpha Buchholz}
	\begin{split}
		I &= \int\limits_0^z t^{b-1}e^{-t}\Phi^2(a^*,b,t)dt = -z^b e^{-z} \frac{\partial \Phi(a^*,b,z)}{\partial z} \frac{\partial \Phi(a^*,b,z)}{\partial a^*} \\[1ex]
		&= -\frac{a^* z^b}{b} e^{-z} \Phi(a^* + 1, b + 1, z) \frac{\partial \Phi(a^*,b,z)}{\partial a^*}
	\end{split}
\end{equation}\\
\indent When $z$ changes by a small amount $\varepsilon_z$, $a^*$ must also change by the amount $\varepsilon_{a^*} = (\partial a^* / \partial z) \varepsilon_z$ to ensure that $\Phi(a^* + \varepsilon_{a^*}, b, z+\varepsilon_z) = 0$. Expanding $\Phi(a^*+\varepsilon_{a^*},b,z + \varepsilon_z)$ in a first-order Taylor series,

\begin{equation}
	\begin{split}
		&\Phi(a^* + \varepsilon_{a^*}, b, z + \varepsilon_z) = \Phi(a^*,b,z) + \frac{\partial \Phi}{\partial a^*} \varepsilon_{a^*} + \frac{\partial \Phi}{\partial z} \varepsilon_z \\[1ex]  &\quad\quad\quad = \Phi(a^*,b,z) + \left( \frac{\partial \Phi}{\partial a^*} \frac{\partial a^*}{\partial z} + \frac{\partial \Phi}{\partial z} \right) \varepsilon_z
	\end{split}
\end{equation}\\
\noindent Since $\Phi(a^* + \varepsilon_{a^*}, b, z + \varepsilon_z) = \Phi(a^*,b,z) = 0$ and $\varepsilon_z$ is arbitrary, it must be that

\begin{equation}
	\label{eq:zero movement PDE}
	 \frac{\partial \Phi}{\partial a^*} \frac{\partial a^*}{\partial z} + \frac{\partial \Phi}{\partial z} = 0
\end{equation}\\
\noindent Eliminating $\partial \Phi / \partial a^*$ from (\ref{eq:4 alpha Buchholz}) and (\ref{eq:zero movement PDE}) yields

\begin{equation}
	\label{eq:del_a/del_z}
	\frac{\partial a^*}{\partial z} = \left( \frac{a^*}{b} \right)^2 z^b e^{-z} \frac{\Phi^2(a^*+1,b+1,z)}{\int_0^z t^{b-1} e^{-t} \Phi^2(a^*,b,t)dt}
\end{equation}\\
Notice that $\partial a^* / \partial z$ is always nonnegative, indicating that as $z$ increases, $a^*$ also increases (moves closer to the origin).

Additional insights are gained when we consider the asymptotic behavior of $\partial a^* / \partial z$ as $z \rightarrow 0$ and $z \rightarrow \infty$. From \cite[Eq. (13.5.5)]{Abramowitz}, $\Phi(a^*,b,z) \rightarrow 1$ as $z \rightarrow 0$, which simplifies the integral in (\ref{eq:del_a/del_z}) to $\int_0^z t^{b-1}e^{-t} dt = \gamma(b,z) = b^{-1}z^{b} \Phi(b,1+b,-z)$ \cite[Eq. (6.5.12)]{Abramowitz}. Therefore,

\begin{equation}
	\lim_{z\rightarrow 0} \left(\frac{\partial a^*}{\partial z}\right) = \lim_{z\rightarrow 0} \left( \frac{a^*}{b} \right)^2 \frac{z^b e^{-z}}{b^{-1}z^{b} \Phi(b,1+b,-z)} = \frac{(a^*)^2}{b}
\end{equation}\\
Since all zeros must decrease as $z$ decreases, $a^* \rightarrow -\infty$ as $z\rightarrow 0$. If this were not the case, it would imply that there are real, finite solutions to  $\Phi(a^*,b,0) = 0$, which is certainly not true. Thus, $\partial a^* / \partial z \rightarrow \infty$ as $z \rightarrow 0$. Now let's analyze the behavior of $\partial a^* / \partial z$ as $z \rightarrow \infty$. To simplify the analysis, we will determine an upper bound on $\partial a^* / \partial z$ by first deriving a lower bound on the integral in (\ref{eq:del_a/del_z}).

For $a^* < 0$ and $(b,z) > 0$, and with $k=b/2 - a^*$, the following inequality holds

\begin{equation}
	\int\limits_0^z k t^{b-1} e^{-t} \Phi^2(a^*,b,t) dt \geq \int\limits_0^z \left( \frac{k}{t} - \frac{1}{2} \right) e^{-t} t^b \Phi^2(a^*,b,t) dt
\end{equation}\\
Replacing the right-hand side with the result from (\ref{eq:integral 2}) and noting that $\Phi(a^*,b,z) = 0$, we get

\begin{equation}
	\int\limits_0^z t^{b-1} e^{-t} \Phi^2(a^*,b,t) dt \geq \frac{z^{b+1} e^{-z}}{k} \left(\frac{a^*}{b}\right)^2 \Phi^2(a^*+1,b+1,z)
\end{equation}\\
Substituting back into (\ref{eq:del_a/del_z}) yields the upper bound $\partial a^* / \partial z \leq k/z$, which tends to zero as $z \rightarrow \infty$. We can therefore conclude that the qualitative behavior of $a^*(z)$ is as shown in Fig. \ref{fig:example zero trajectories}.

\begin{figure}[!htb]
	\centering
	\includegraphics[scale=0.8]{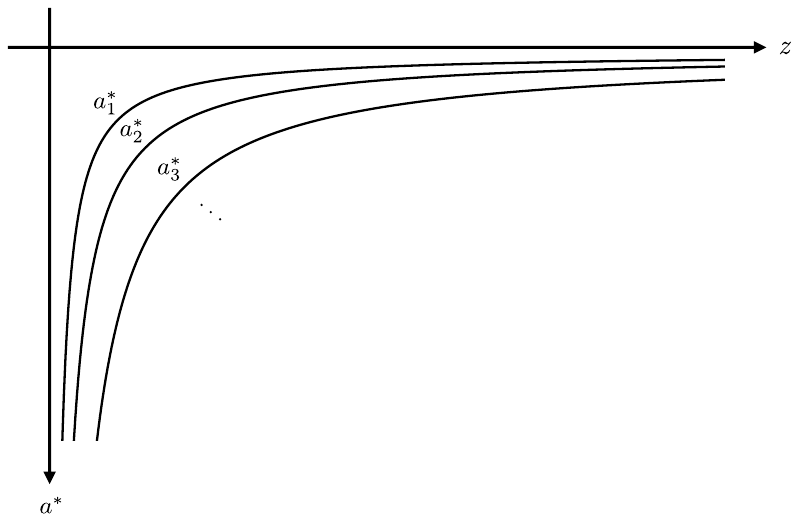}
	\caption{Qualitative depiction of trajectories followed by the zeros $a^*$ as $z$ varies.}
	\label{fig:example zero trajectories}
\end{figure}

The curves depicted in Fig. \ref{fig:example zero trajectories} have two important features. First, none of them intersect, which follows from the fact that all zeros $a^*$ must be simple. Any intersections would imply the existence of zeros with multiplicity greater than one. The second feature is that a given curve $a^*(z)$ is continuous. To show that this must be true, first observe that any point $(a^*,z)$ satisfying $\Phi(a^*,b,z) = 0$ is a regular point because all zeros $a^*$ are simple and thus $\nabla \Phi \neq 0$ at $(a^*,z)$ \cite[Prop. 8.23]{Tu}. This implies that $0$ is a regular value of the map $\Phi : \mathbb{R}^2 \rightarrow \mathbb{R}$ and that the level set $\Phi^{-1}(0)$ is also regular \cite[p. 103]{Tu}. In addition, we note that $\Phi : \mathbb{R}^2 \rightarrow \mathbb{R}$ is a $C^{\infty}$ map since $\Phi(a,b,z)$ is an entire function of $a$ and $z$. The regular level set theorem \cite{Tu} then asserts that $\Phi^{-1}(0)$ must be a regular submanifold of $\mathbb{R}$, i.e., each curve in Fig. \ref{fig:example zero trajectories} must be smooth. These properties of $a^*(z)$ allow us to conclude that (\ref{eq:del_a/del_z}) has a unique, continuous solution and leads to the following proposition.

\begin{proposition}\label{prop:a-domain zero spacing}
	Let $a_{k+1}^* < a_k^*$ be two consecutive zeros of $\Phi(a,b,z_l)$ for the given values $(b, z_l) \in \mathbb{R}^+$. Suppose that $a_k^*$ lies in the interval $[-N,-N+1]$ for some positive integer $N$ so that there is a sequence of $N$ values $0<z_1<z_2 < \cdots < z_N$, of which $z_l$ is a member, that satisfy the equation $\Phi(a_k^*,b,z_n) = 0$, $n=1,\dots,N$. Now let $a^*(z)$ be the solution to the initial value problem (IVP)
	
	\begin{equation}
		\label{eq:IVP}
		\frac{\partial a^*}{\partial z} = \left( \frac{a^*}{b} \right)^2 z^b e^{-z} \frac{\Phi^2(a^*+1,b+1,z)}{\int_0^z t^{b-1} e^{-t} \Phi^2(a^*,b,t)dt}, \quad a^*(z_l) = a^*_{k+1}
	\end{equation}\\
	\noindent Then when $z_l \neq z_N$, $a^*(z_{l+1}) = a_k^*$. If $z_l = z_N$, $a^*(z) \rightarrow -N$ as $z \rightarrow \infty$.
\end{proposition}

\begin{proof}
	Let's focus first on the case where $z_l \neq z_N$. Consider the diagram in Fig. \ref{fig:zero movement}, showing the trajectories of two consecutive zeros as a function of $z$. Notice that when $z$ increases from $z_l$ to $z_{l+1}$, the curve $a^*(z)$ increases from $a_{k+1}^*$ to $a_k^*$. Since the evolution of $a^*(z)$ is governed by the differential equation in (\ref{eq:del_a/del_z}), it must be then that $a_k^*$ is the solution to the IVP in (\ref{eq:IVP}) at $z=z_{l+1}$.
	
	\begin{figure}[!htb]
		\centering
		\includegraphics[scale=0.7]{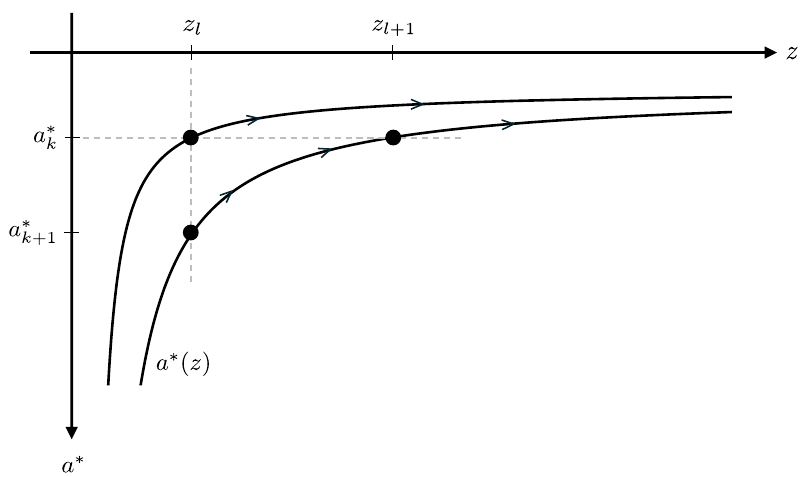}
		\caption{Trajectories of two consecutive zeros.}
		\label{fig:zero movement}
	\end{figure}
	
	Now suppose that $z_l = z_N$, so that no zero $z > z_N$ exists for which $\Phi(a_k^*,b,z) = 0$. That is, there is no amount of increase in $z$ such that $a^*(z) = a_k^*$. The limiting position of $a^*(z)$ can be determined by analyzing the asymptotic behavior of $a^*(z)$ for large $z$. From \cite[Eq. (13.1.4)]{Abramowitz}, we have for $z>>1$,
	
	\begin{equation}
		\Phi(a,b,z) = \frac{\Gamma(b)}{\Gamma(a)} e^z z^{a-b} [1+\mathcal{O}(z^{-1})]
	\end{equation}\\
	\noindent As $z$ grows larger, the only way for $\Phi(a,b,z)$ to vanish is for $\Gamma(a)$ to also grow large, which occurs as $a$ approaches a pole of the gamma function at one of the negative integers. This implies that when $z_l = z_N$, so that $a_k^* \in [-N,-N+1]$ for some positive integer $N$, $a^*(z)$ will approach $-N$ as $z \rightarrow \infty$. It is not possible for $a^*(z)$ to settle at some other integer greater than $-N$, since this would require $a^*(z)$ to pass through $a_k^*$ for some finite $z$, violating the fact that there is no $z>z_l$ for which $a_k^*$ is a zero.
\end{proof}

Proposition \ref{prop:a-domain zero spacing} provides a link between the spacing of zeros in the $z$-domain and the spacing of zeros in the $a$-domain, and enables us to determine a lower bound on $a_{k}^* - a_{k+1}^*$ through Proposition \ref{prop:z-domain zero spacing}. That is, by solving (\ref{eq:IVP}) up to $\bar{z} < z_{l+1}$ (if $z_{l+1}$ exists, otherwise we only require $\bar{z}<\infty$), the resulting solution $\bar{a}_k^*$ will be less than $a_k^*$, which implies that $\bar{a}_k^* - a_{k+1}^* < a_k^* - a_{k+1}^*$. However, because (\ref{eq:IVP}) has no analytical solution, an explicit expression cannot be written for $\bar{a}_k$. This makes it difficult to formulate general statements about the behavior of $\bar{a}_k^* - a_{k+1}^*$ and limits the utility of the bound. We therefore seek to approximate (\ref{eq:IVP}) so that an analytic solution is achievable.

\section{Determining an Analytic Bound}\label{sec:analytic bound}

We first leverage the following comparison theorem, proved in \cite{Budincevic}.

\begin{proposition}\footnote[4]{Reference \cite{Budincevic} provides a weaker version of this theorem when certain uniqueness or Lipschitz continuity conditions are met, but it is not required for our purposes.}\label{prop:comparison theorem} Suppose that the functions $f(t,y)$ and $h(t,y)$ are continuous in the domain
	
	\begin{equation*}
		D = \left\{ (t,y) : \lvert t - t_0 \rvert < c, \lvert y - y_0 \rvert < d \right\},
	\end{equation*}\\	
	\noindent and denote by $y(t)$, $v(t)$ any solution of the IVPs\\[-1.5ex]
	
	\quad (1) $y^{\prime}(t) = f(t,y), y(t_0) = y_0$ \\[-1.5ex]
		
	\quad (2) $v^{\prime}(t) = h(t,v), v(t_0) = y_0$ \\[-1.5ex]
	
	\noindent respectively. If $h(t,y) < f(t,y)$ in $D$, then $v(t) < y(t)$ for $t>t_0$.
\end{proposition}

\noindent Let $f(a^*,z)$ be the right-hand side of the ODE in (\ref{eq:IVP}), i.e.,

\begin{equation}
	\label{eq:del_a/del_z copy}
	f(a^*,z) = \left( \frac{a^*}{b} \right)^2 z^b e^{-z} \frac{\Phi^2(a^*+1,b+1,z)}{\int_0^z t^{b-1} e^{-t} \Phi^2(a^*,b,t)dt}
\end{equation}\\
\noindent Proposition \ref{prop:comparison theorem} says that if $f(a^*,z)$ is replaced with a lower bound $h(a^*,z)$, the resulting solution $\bar{a}^*(z)$ to the IVP will be less than $a^*(z)$ for all $z>z_l$. One way to obtain $h(a^*,z)$ is to upper bound the integral in (\ref{eq:del_a/del_z copy}). To accomplish this, let's first develop an alternative expression for the integral.

\begin{proposition} Consider real parameters $a^*<0$, $b>0$ and $z>0$ such that $\Phi(a^*,b,z) = 0$. Then the integral $I = \int_0^z t^{b-1} e^{-t} \Phi^2(a^*,b,t) dt$ can also be written as
	
	\begin{equation}
		\label{eq:integral recursion}
		I = \frac{(b-1)^2}{-1-a^*+b} \int\limits_0^z t^{b-2} e^{-t} \Phi^2(a^*, b-1, t) dt
	\end{equation}
\end{proposition}

\begin{proof}
\cite[Eq. (13.4.3)]{Abramowitz}	gives the recurrence relation

\begin{equation}
	\Phi(a,b,t) = \frac{a}{1 + a - b} \Phi(a+1,b,t) - \frac{b-1}{1+a-b} \Phi(a,b-1,t)
\end{equation}\\
\noindent Substituting into the definition of $I$ yields

\begin{equation}
	\label{eq:I integral}
	\begin{split}
		I &= \frac{a}{1+a-b} \int\limits_0^z t^{b-1} e^{-t} \Phi(a,b,t) \Phi(a+1,b,t) dt \\[1ex] &+ \frac{b-1}{-1-a+b} \int\limits_0^z t^{b-1} e^{-t} \Phi(a,b,t) \Phi(a,b-1,t) dt
	\end{split}
\end{equation}\\
\noindent Let's write (\ref{eq:I integral}) as $I = I_1 + I_2$. Using (\ref{eq:integral 1}) with $\xi = a$, $\eta = a + 1$, and the fact that $\Phi(a,b,z) = 0$, it is straightforward to show that

\begin{equation}
	\label{eq:I_1 integral}
	I_1 = \frac{a^2 e^{-z} z^{b-1}}{-1-a+b} \Phi^2(a+1,b,z)
\end{equation}\\
\noindent For $I_2$, use \cite[Eq. (13.4.4)]{Abramowitz} to write

\begin{equation}
	\Phi(a,b,t) = \frac{b-1}{t} \Phi(a,b-1,t) - \frac{b-1}{t} \Phi(a-1,b-1,t)
\end{equation}\\
\noindent which results in

\begin{equation}
	\label{eq:I_2 integral}
	\begin{split}
		I_2 &= \frac{(b-1)^2}{-1-a+b} \int\limits_0^z t^{b-2} e^{-t} \Phi^2(a,b-1,t) dt\\[1ex] &- \frac{(b-1)^2}{-1-a+b} \int\limits_0^z t^{b-2} e^{-t} \Phi(a-1,b-1,t) \Phi(a,b-1,t) dt
	\end{split}
\end{equation}\\
\noindent Applying (\ref{eq:integral 1}) to the second integral in (\ref{eq:I_2 integral}) with $\xi=a-1$, $\eta=a$ and again using the fact that $\Phi(a,b,z) = 0$, we get

\begin{equation}
	\label{eq:I_2 integral simplified}
	I_2 = \frac{(b-1)^2}{-1-a+b} \int\limits_0^z t^{b-2} e^{-t} \Phi^2(a,b-1,t) dt - \frac{a^2 e^{-z} z^{b-1}}{-1-a+b} \Phi^2(a+1,b,z)
\end{equation}\\
\noindent Substituting (\ref{eq:I_1 integral}) and (\ref{eq:I_2 integral simplified}) into (\ref{eq:I integral}), and replacing $a$ with $a^*$ yields the desired result.
\end{proof}

\begin{lemma}\label{lem:integral_bound} Consider real parameters $a<0$, $b>0$ and $z>0$ such that $\Phi(a,b,z) = 0$ and $z<-1-a+b$. Then an upper bound on the integral $I=\int_0^z t^{b-1} e^{-t} \Phi^2(a,b,t)dt$ is
	
\begin{equation}
	\label{eq:integral bound}
	I\leq \frac{a^2 e^{-z} z^{b-1} \Phi^2(a+1,b,z)}{-1 - a + b - \sqrt{z(-1-a+b)}}
\end{equation}
\end{lemma}

\begin{proof}
	Return to (\ref{eq:I integral}) and substitute the result from (\ref{eq:I_1 integral}) for $I_1$. Then
	
	\begin{equation}
		\label{eq:bound on I}
		I \leq \frac{a^2 e^{-z} z^{b-1}}{-1-a+b} \Phi^2(a+1,b,z) + J
	\end{equation}\\	
	\noindent with
	
	\begin{equation}
		J = \left \lvert \frac{b-1}{-1-a+b} \int\limits_0^z \left[t^{(b-1)/2}e^{-t/2} \Phi(a,b,t)\right] \left[t^{(b-1)/2}e^{-t/2} \Phi(a,b-1,t)\right] dt \right \rvert
	\end{equation}\\	
	\noindent Use the Cauchy-Schwarz inequality to write
	
	\begin{equation}
		\label{eq:Cauchy Schwarz}
		\begin{split}
			J &\leq \left[ \left(\frac{b-1}{-1-a+b}\right)^2 \int\limits_0^z t^{b-1} e^{-t} \Phi^2(a,b,t) dt \right]^{1/2} \left[ \int\limits_0^z t^{b-1} e^{-t} \Phi^2(a,b-1,t) dt \right]^{1/2}\\[1ex] &= \sqrt{J_1} \cdot \sqrt{J_2}
		\end{split}
	\end{equation}\\	
	\noindent The first integral $J_1$ is recognized as $\left(\frac{b-1}{-1-a+b}\right)^2 I$. For $J_2$, first write it as
	
	\begin{equation}
		J_2 = \int\limits_0^z t \cdot t^{b-2} e^{-t} \Phi^2(a,b-1,t) dt
	\end{equation}\\	
	\noindent Given that $z>0$ and that the function $t$ is monotonic over the interval $[0,z]$, an upper bound on $J_2$ is \cite{Underhill}
	
	\begin{equation}
		J_2 \leq z \int\limits_0^z t^{b-2} e^{-t} \Phi^2(a,b-1,t)dt \leq z \frac{-1-a+b}{(b-1)^2} I
	\end{equation}\\	
	\noindent where we used the result from (\ref{eq:integral recursion}). Therefore, after substituting back into (\ref{eq:Cauchy Schwarz}), we get
	
	\begin{equation}
		\label{eq:J bound}
		J \leq I \sqrt{\left(\frac{b-1}{-1-a+b}\right)^2} \sqrt{z\frac{-1-a+b}{(b-1)^2}} = I \sqrt{\frac{z}{-1-a+b}}
	\end{equation}\\	
	\noindent Substituting (\ref{eq:J bound}) back into (\ref{eq:bound on I}) yields the inequality
	
	\begin{equation}
		I \leq \frac{a^2 e^{-z} z^{b-1}}{-1-a+b} \Phi^2(a+1,b,z) + I\sqrt{\frac{z}{-1-a+b}}
	\end{equation}\\	
	\noindent which we can simplify into the form
	
	\begin{equation}
		\label{eq:I bound}
		I \left(1 - \frac{\sqrt{z}}{\sqrt{-1-a+b}}\right) \leq \frac{a^2 e^{-z} z^{b-1}}{-1-a+b} \Phi^2(a+1,b,z)
	\end{equation}\\	
	\noindent Provided that $z < -1-a+b$, rearranging (\ref{eq:I bound}) yields the upper bound in (\ref{eq:integral bound}).
\end{proof}

\subsection{Lower Bound on Zero Separation}

With an upper bound on $I$, we can state the following theorem.
\FirstTheorem*

\begin{proof}
	Substituting the upper bound in (\ref{eq:integral bound}) for the integral in (\ref{eq:del_a/del_z copy}) and using \cite[Eq. (13.4.4)]{Abramowitz}, we get the following lower bound on $f(a^*,z)$
	
	\begin{equation}
		\label{eq:h(a^*,z)}
		h(a^*,z) = \frac{-1-a^*+b}{z} - \frac{\sqrt{-1-a^*+b}}{\sqrt{z}}
	\end{equation}\\
	\noindent Leveraging the comparison theorem in Proposition \ref{prop:comparison theorem}, (\ref{eq:h(a^*,z)}) allows us to consider a much simpler differential equation when analyzing the spacing between consecutive zeros, namely,
	
	\begin{equation}
		\label{eq:del_a_star_del_z}
		\frac{\partial a^*}{\partial z} = \frac{-1-a^*+b}{z} - \frac{\sqrt{-1-a^*+b}}{\sqrt{z}}
	\end{equation}\\
	\indent With the initial condition $a^*(z_l) = a_k^*$, it is straightforward to show using separation of variables that the solution to (\ref{eq:del_a_star_del_z}) is
	
	\begin{equation}
		\label{eq:IVP solution}
		a^*(z) = -1 + b - \frac{1}{4} \left[ \left(\frac{z_l}{z}\right)^{1/2} \left(2\sqrt{-1-a_k^*+b} - \sqrt{z_l} \right) + \sqrt{z} \right]^2
	\end{equation}\\
	\noindent Let $z_{l+1}$ be the next value (assuming it exists) for which $a_{k-1}^*$ is a zero. Then the solution $a^*(z_{l+1})$ to (\ref{eq:IVP solution}) is a lower bound on $a_{k-1}^*$. We want to avoid computing $z_{l+1}$ because this would require us to first compute $a_{k-1}^*$, which nullifies the need to obtain a bound on $a_{k-1}^* - a_k^*$. Recall from the discussion following Proposition \ref{prop:a-domain zero spacing} that if $z_{l+1}$ is replaced with a lower bound $\bar{z}$, then $a^*(\bar{z})$ will be a lower bound on $a_{k-1}^*$. Using Proposition \ref{prop:z-domain zero spacing}, $\bar{z}$ is given by\footnote[5]{Proposition \ref{prop:z-domain zero spacing} requires $b - a^* - 1 > 0$, which is automatically satisfied by the condition $b - a^* - 1 > z$ needed in Lemma \ref{lem:integral_bound}.}
	
	\begin{equation}
		\label{eq:z bar}
		\bar{z} = z_l \exp{\left[ \frac{2\pi}{\sqrt{\left(b - 2 a_{k-1}^*\right)^2 - \left(b-1\right)^2}}  \right]} = z_l g_{k-1}
	\end{equation}\\
	\noindent Prior to substituting $\bar{z}$ for $z$ in (\ref{eq:IVP solution}), note that it is permissible to replace $a_{k-1}^*$ in (\ref{eq:z bar}) with $a_k^*$ since this has the effect of reducing $\bar{z}$. We will perform this replacement because it ensures that $a_k^*$ is the only zero that appears on the right-hand side of (\ref{eq:IVP solution}) and it will also simplify the monotonicity analysis in Section \ref{sec:monotonicity}.
	
	After substituting $\bar{z} = z_l g_k$ for $z$ in (\ref{eq:IVP solution}) and subtracting $a_k^*$ from both sides, we obtain the desired bound on $\Delta a = a_{k-1}^* - a_k^*$
	
	\begin{equation}
		\label{eq:bound on delta a}
		\Delta a \geq \beta_k - \frac{\beta_k}{4 g_k} \left[ 2 + \sqrt{z_l/\beta_k} (g_k - 1) \right]^2
	\end{equation}\\
	\noindent The last step is to prove the condition $z_l < \beta_{k-1}/g_{k-1}$. Recall that a key requirement of the upper bound in (\ref{eq:integral bound}) was that $z<-1-a^*+b$. This inequality must be valid over the entire solution space of the differential equation in (\ref{eq:del_a_star_del_z}). That is, for any $z \in [z_l, z_l g_k]$ and $a^* \in [a_k^*, a_{k-1}^*]$. To ensure that $z < -1-a^*+b$ is satisfied everywhere, replace the left-hand side of the inequality with an upper bound and the right-hand side with a lower bound. Given that $a_k^* < a_{k-1}^* < 0$ and the definition of $g_k$ from (\ref{eq:z bar}), for any $z \in [z_l, z_l g_k]$, $z \leq z_l g_{k-1}$. Similarly, for any $a^* \in [a_k^*, a_{k-1}^*]$, $(-1 - a^* + b) \geq (-1 - a_{k-1}^*+b) = \beta_{k-1}$. Therefore, if $z_l < \beta_{k-1} / g_{k-1}$, the integral upper bound in (\ref{eq:integral bound}) will be valid for all $z$ and $a^*$ in their respective domains.
\end{proof}

\section{Monotonicity of the Bound}\label{sec:monotonicity}

Numerical investigation of (\ref{eq:bound on delta a}) suggests that $\Delta a$ is a monotonically decreasing function of $a_k^*$. It is difficult to prove this statement for all $a_k^*$, but we can derive a tight upper bound $\bar{a}^*$ such that monotonicity holds for $a_k^* < \bar{a}^*$.

\SecondTheorem*

\begin{proof}
	We will show that $d\Delta a / d a_k^* < 0$. Applying the chain rule to (\ref{eq:bound on delta a}) yields
	
	\begin{equation}
		\label{eq:chain rule}
		\frac{d \Delta a}{d a_k^*} = \frac{\partial \Delta a}{\partial \beta_k} \frac{d \beta_k}{d a_k^*} + \frac{\partial \Delta a}{\partial g_k} \frac{dg_k}{d a_k^*} = -\frac{\partial \Delta a}{\partial \beta_k} + \frac{\partial \Delta a}{\partial g_k} \frac{dg_k}{d a_k^*}
	\end{equation}\\
	\noindent With $\eta = \sqrt{z_l/\beta_k}$, it is straightforward to show that
	
	\begin{equation}
		\label{eq:partial derivatives}
		\begin{aligned}
			\frac{\partial \Delta a}{\partial \beta_k} &= 1 - \frac{1}{2g_k} \left[ 2 + \eta (g_k - 1) \right] \\[1ex]
			\frac{\partial \Delta a}{\partial g_k} &= -\frac{\beta_k}{2g_k} \eta \left[ 2 + \eta(g_k - 1) \right] + \frac{\beta_k}{4 g_k^2} \left[ 2 + \eta (g_k - 1) \right]^2
		\end{aligned}
	\end{equation}\\
	\noindent From the definition of $g_k$ in (\ref{eq:z bar}), we also have
	
	\begin{equation}
	\label{eq:dg da_star}
		\frac{d g_k}{d a_k^*} = \frac{4\pi (b - 2 a_k^*)}{\left[ (b - 2a_k^*)^2 - (b - 1)^2 \right]^{3/2}} g_k = \psi g_k
	\end{equation}\\
	\indent Let's define $f = 2 + \eta (g_k - 1)$. Then after substituting (\ref{eq:partial derivatives}) and (\ref{eq:dg da_star}) into (\ref{eq:chain rule}), we get
	
	\begin{equation}
		\label{eq:d del_a da_star}
		\frac{d \Delta a}{d a_k^*} = - \left[1 - \frac{f}{2g_k} + \frac{f}{2g_k} \left(\eta - \frac{f}{2g_k}\right) \beta_k \psi g_k \right]
	\end{equation}\\
	\noindent Notice that $\beta_{k-1} / g_{k-1} \leq \beta_k$, which follows from the fact that $g_{k-1} \geq 1$ and $\beta_{k-1} < \beta_k$ because $a_k^* < a_{k-1}^*$. Therefore, since $z_l < \beta_{k-1} / g_{k-1}$, we also have $z_l < \beta_k$, leading to the conclusion that $0 \leq \eta \leq 1$. With this in mind, we can show that $f/(2g_k) \leq 1$ as follows
	
	\begin{equation}
		\label{eq:f/2g inequality}
		\frac{f}{2g_k} = \frac{1}{g_k} + \frac{\eta}{2} - \frac{\eta}{2g_k} \leq 1 \Rightarrow \frac{1}{g_k} \left(1 - \frac{\eta}{2}\right) \leq 1 - \frac{\eta}{2} \Rightarrow g_k \geq 1
	\end{equation}\\
	\noindent Since $g_k$ is indeed greater than or equal to one, the inequalities in (\ref{eq:f/2g inequality}) are valid, i.e., $f/(2g_k) \leq 1$. From the condition $b-a-1 > 0$ in Proposition \ref{prop:z-domain zero spacing}, we see that $\beta_k > 0$. In addition, for $\psi$ in (\ref{eq:dg da_star}), notice that that the numerator is always positive since $a_k^* <0$ and that the denominator is positive for $a < b-1/2$, which is always true because $b > 0$ and $b - a - 1 > 0$ (i.e., $a < b - 1$ implies $a < b-1/2$). Thus, $\psi$ is also guaranteed to be positive.
	
	Now let's focus on the term $D$ in square brackets in (\ref{eq:d del_a da_star}). One condition for $D$ to be guaranteed positive (and thus $d\Delta a / d a_k^*$ is guaranteed negative) is if
	
	\begin{equation}
		\label{eq:D inequality}
		\left(\eta - \frac{f}{2g_k}\right) \beta_k \psi g_k \geq -1 \Rightarrow \eta \geq 2 \frac{1 - \displaystyle \frac{1}{\beta_k \psi}}{g_k + 1}
	\end{equation}\\
	\noindent We know that $\eta$ must be nonnegative. Thus, (\ref{eq:D inequality}) will always be satisfied when $\beta_k \psi \leq 1$. Substituting the definitions for $\beta_k$ and $\psi$,
	
	\begin{equation}
		\label{eq:beta_k psi}
		\beta_k \psi = \frac{4\pi(-1 - a_k^* + b)(b - 2a_k^*)}{[(b - 2a_k^*) ^ 2 - (b - 1)^2]^{3/2}}
	\end{equation}\\
	\noindent Notice that as $a_k^* \rightarrow -\infty$, $\beta_k \psi \rightarrow 0$. Thus, if we can find the smallest value for $a_k^*$, call it $\bar{a}^*$, for which $\beta_k \psi = 1$, then it must be that $\beta_k \psi < 1$ for all $a_k^* \leq \bar{a}^*$.
	
	To find $\bar{a}^*$, first define $y^2 = (b - 2 \bar{a})^2 - (b-1)^2$ so that
	
	\begin{equation}
		\label{eq:a_bar}
		\bar{a} = \frac{b}{2} \pm \frac{1}{2} \sqrt{(b-1) ^ 2 + y ^ 2}
	\end{equation}\\
	\noindent Focusing on the "$-$" solution in (\ref{eq:a_bar}), substitution into (\ref{eq:beta_k psi}) and rearranging terms yields
	
	\begin{equation}
		\label{eq:intermediate expression}
		\sqrt{(b-1)^2 + y^2} = \frac{\displaystyle \frac{y^3}{2\pi} - (b-1)^2 - y^2}{b-2}
	\end{equation}\\
	\noindent Squaring both sides of (\ref{eq:intermediate expression}) and simplifying yields the sixth-order polynomial
	
	\begin{equation}
		\label{eq:sixth order poly}
		\frac{1}{4\pi^2}y^6 - \frac{1}{\pi} y^5 + y^4 - \frac{(b-1)^2}{\pi} y^3 + (b^2 - 2) y^2 + (b-1)^2(2b-3) = 0
	\end{equation}\\
	\noindent The same polynomial is obtained for the "$+$" solution in (\ref{eq:a_bar}). All six roots can easily be found using routine numerical algorithms and substituted back into (\ref{eq:a_bar}) to determine the corresponding values $\bar{a}$, of which we are only interested in real solutions. Because of the squaring operation between (\ref{eq:intermediate expression}) and (\ref{eq:sixth order poly}), not all of the $\bar{a}$'s satisfy $\beta_k \psi = 1$ and feasibility needs to be verified. Then $\bar{a}^*$ is the minimum of the set of real and feasible $\bar{a}$'s and the theorem is proved.
\end{proof}

\subsection{Discussion}

In this subsection we analyze the behavior of $\bar{a}^*$ as a function of $b$. First, we point out that for $b < 0.32$, $\beta_k \psi$ is always less than $1$, meaning that the bound in (\ref{eq:bound on delta a}) is monotonic over the entire domain of $a_k^*$. This conclusion is reached by determining that there are no solutions $\bar{a}^*$ to $\beta_k \psi = 1$ when $b<0.32$, thereby making it impossible to satisfy the equality constraint in Theorem \ref{secondary theorem}. Figure \ref{fig:large b monotonicity} shows the values for $\bar{a}^*$ for $b \geq 0.32$. The key observation from Fig. \ref{fig:large b monotonicity} is that the critical value $\bar{a}^*$ is relatively small, even for $b$ as large as $10,000$. Thus, the bound in (\ref{eq:bound on delta a}) is monotonic over much of the negative real axis.

\begin{figure}[!htb]
	\centering
	\includegraphics[scale=0.62]{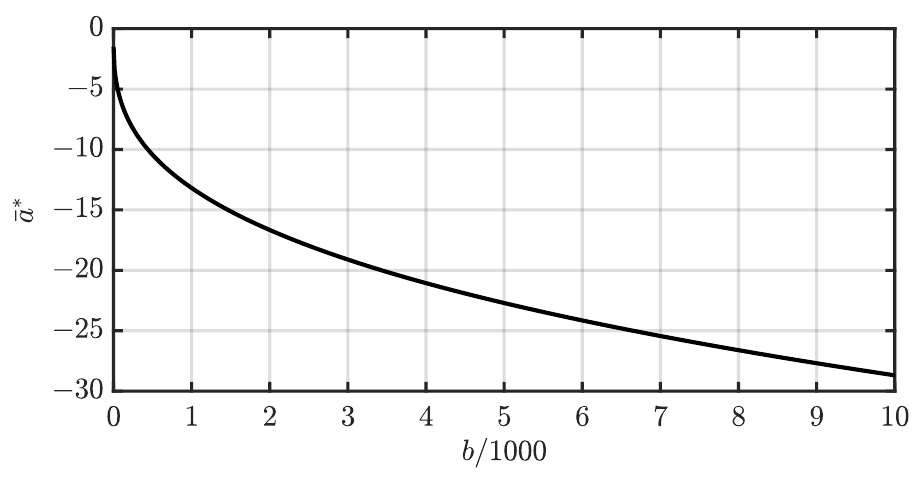}
	\caption{Monotonicity threshold value $\bar{a}^*$ for $b \geq 0.32$.}
	\label{fig:large b monotonicity}
\end{figure}

\section{First Passage Problem}\label{sec:first passage problem}

In this section, we will use the previous results to assess the accuracy of asymptotic approximations for the first passage probability of a Wiener process. For context, a maximum likelihood test was developed in \cite{Vostrikova} to determine when a change in drift has occurred in an $m$-dimensional Wiener process over the dimensionless time interval $[1, n]$\footnote[6]{In \cite{Vostrikova}, the non-dimensional time interval is denoted as $[1, (1-\alpha)^2/\alpha^2]$, where $0<\alpha<0.5$.}. The authors showed that the probability of false alarm, $P_{fa}$, for their test could be expressed in terms of a first passage problem. Specifically, they showed that $P_{fa}$ is equivalent to the probability that the magnitude of a standard, $m$-dimensional Wiener process $\textbf{\emph{x}}(t)$ first crosses a threshold $y\sqrt{t}$ at some time $t \leq n$.

The only way to analytically quantify $P_{fa}$ is as an inverse Laplace transform, i.e., \cite[Eq. (14)]{Vostrikova}

\begin{equation}
	\label{eq:L.T. of Pfa}
	P_{fa} = \frac{\Gamma(m/2, y^2/2)}{\Gamma(m/2)} + \frac{y^m e^{-y^2/2}}{m 2^{m/2-1} \Gamma(m/2)} \mathcal{L}^{-1} \left\{ \frac{\Phi(\nu+1,m/2+1,y^2/2)}{\nu\Phi(\nu, m/2, y^2/2)}\right\}(\ln n)
\end{equation}\\
\noindent An approximate inverse transform is achievable by asymptotically expanding the ratio of hypergeometric functions for large $y$ and retaining the first-order term, resulting in \cite[Eq. (18)]{Vostrikova}

\begin{equation}
	\label{eq:approx Pfa}
	P_{fa} \approx \frac{e^{-y^2/2} y^m}{\Gamma(m/2) 2^{m/2}} \left[\ln n \left(1 - \frac{m}{y^2}\right) + \frac{4}{y^2}\right]
\end{equation}

\subsection{Exact False Alarm Probability in Terms of Residues}
\indent Let the function in curly braces in (\ref{eq:L.T. of Pfa}) be $G(\nu)$ and denote the time-domain variable as $u$. To assess the accuracy of (\ref{eq:approx Pfa}), we will first obtain an exact expression for $P_{fa}$ by evaluating $\mathcal{L}^{-1}\left\{G(\nu)\right\}$ via residues. We prove in Appendix \ref{Appendix B} that the inverse transform can be written as

\begin{equation}
	\label{eq:residues}
	\mathcal{L}^{-1}\left\{G(\nu)\right\} = \text{Res}(e^{\nu u} G, 0) + \sum_{k=1}^{\infty} \text{Res} (e^{\nu u} G, \nu_k^*)
\end{equation}\\
\noindent where $\text{Res}(e^{\nu u} G, \nu_k^*)$ is the residue of $e^{\nu u}G(\nu)$ at the pole $\nu_k^*$ and $\nu_k^* < \cdots < \nu_1^* < 0$ are the nontrivial poles of $G(\nu)$, i.e., the zeros of $\Phi(\nu,m/2,y^2/2)$. Notice that all poles of $e^{\nu u} G$ are real and simple (Proposition \ref{prop:zeros a_star of Phi}).

For the simple pole at $\nu = 0$, we have

\begin{equation}
	\label{eq:residue at zero}
	\begin{aligned}
		\text{Res}(e^{\nu u}G, 0) &= \lim_{\nu\rightarrow 0} \nu e^{\nu u} G = \lim_{\nu\rightarrow 0} \nu e^{\nu u} \frac{G(\nu+1, m/2+1, y^2/2)}{\nu\Phi(\nu, m/2, y^2/2)} \\[1ex] &= \Phi(1, m/2+1, y^2/2)
	\end{aligned}
\end{equation}\\
\noindent Recognizing that $e^{\nu u} G$ is a ratio of functions, the residue for all other simple poles $\nu_k^* \neq 0$ is \cite{Kapoor}

\begin{equation}
	\text{Res}(e^{\nu u}G, \nu_k^*) = \left. \frac{(e^{\nu u}/\nu) \Phi(\nu+1, m/2+1, y^2/2)}{d\Phi(\nu, m/2, y^2/2) / d\nu} \right|_{\nu = \nu_k^*} \;\; , \;\; \nu_k^* \neq 0
\end{equation}\\
\noindent Since $\Phi(\nu_k^*, m/2, y^2/2) = 0$, we can use (\ref{eq:4 alpha Buchholz}) to evaluate $d\Phi/d\nu$, leading to the expression

\begin{equation}
	\label{eq:residue nontrivial poles}
	\text{Res}(e^{\nu u}G, \nu_k^*) = -\frac{(y^2/2)^{m/2} e^{\nu_k^* u}}{(m/2) e^{y^2/2}} \frac{\Phi^2(\nu_k^* + 1, m/2 + 1, y^2/2)}{\displaystyle\int_0^{y^2/2} t^{m/2-1} e^{-t} \Phi^2(\nu_k^*, m/2, t)dt}
\end{equation}\\
\indent Substituting (\ref{eq:residue at zero}) and (\ref{eq:residue nontrivial poles}) back into (\ref{eq:residues}), $P_{fa}$ in (\ref{eq:L.T. of Pfa}) can now be written as

\begin{equation}
	\label{eq:final Pfa}
	P_{fa} = \frac{\Gamma(b, z)}{\Gamma(b)} + \frac{z^b e^{-z}}{b\Gamma(b)} \Phi(1, b+1, z) - \frac{z^{2b} e^{-2z}}{b^2 \Gamma(b)} \sum\limits_{k=1}^{\infty} \frac{n^{\nu_k^*} \Phi^2(\nu_k^* + 1, b + 1, z)}{\int_0^z t^{b-1} e^{-t} \Phi^2(\nu_k^*,b,t) dt}
\end{equation}
\\ where $b=m/2$ and $z = y^2/2$.

\subsection{Bounding the False Alarm Probability}

In this section, guaranteed bounds on $P_{fa}$ are derived that can be used to assess the accuracy of (\ref{eq:approx Pfa}). It is straightforward to obtain an upper bound $P_{fa}^{(u)}$ by truncating the series in (\ref{eq:final Pfa}) to $N$ terms because the contribution of each term in the sum to $P_{fa}$ is negative. Therefore, we can write $P_{fa} = P_{fa}^{(u)} - \varepsilon_N$, where

\begin{equation}
	\label{eq:truncation error}
	\varepsilon_N = \frac{z^{2b}e^{-2z}}{b^2 \Gamma(b)} \sum\limits_{k=N+1}^{\infty} n^{\nu_k^*} \frac{\Phi^2(\nu_k^* + 1, b + 1, z)}{\int_0^z t^{b-1} e^{-t} \Phi^2(\nu_k^*, b, t)dt}
\end{equation}\\
\noindent Since the truncation error $\varepsilon_N$ is positive (each term in (\ref{eq:truncation error}) is positive), given an upper bound $\bar{\varepsilon}_N$ we can immediately construct the lower bound $P_{fa}^{(l)} = P_{fa}^{(u)} - \bar{\varepsilon}_N$.

\begin{proposition}\label{prop:eps_N bound} Let $b$ and $z$ be positive real numbers and $-\infty < \cdots < \nu_{k+1}^* < \nu_k^* < \cdots < 0$ be the sequence of zeros of $\Phi(\nu,b,z)$. Given a bound $\Delta$ such that $\nu_k^* - \nu_{k+1}^* \geq \Delta$ for $k>N$, an upper bound on $\varepsilon_N$ is
	
	\begin{equation}
		\label{eq:eps_N bound}
		\bar{\varepsilon}_N = \frac{z^{b-1} e^{-z} (b-2\nu_N^*)}{2 (\nu_N^*)^2 \Gamma(b)} \frac{n^{\nu_N^*}}{n^{\Delta} - 1}
	\end{equation}
\end{proposition}

\begin{proof}
	To get an upper bound on $\varepsilon_N$, the first step is to derive a lower bound on the integral $I$ in (\ref{eq:truncation error}). From (\ref{eq:integral 2}) with $a = \nu_k^*$, and noting that $\Phi(\nu_k^*, b, z) = 0$,
	
	\begin{equation}
		\label{eq:integral result}
		I = \frac{2}{b-2\nu_k^*} z^{b+1} e^{-z} \left(\frac{\nu_k^*}{b}\right)^2 \Phi^2(\nu_k^*+1, b+1, z) + \frac{1}{b-2\nu_k^*} \int\limits_0^z e^{-t} t^b \Phi^2(\nu_k^*, b, t) dt
	\end{equation}\\
	\noindent The second term on the right-hand side of (\ref{eq:integral result}) is nonnegative. Thus, a lower bound on $I$ is obtained by ignoring this term, i.e.,
	
	\begin{equation}
		\label{eq:lower bound}
		I \geq \frac{2(\nu_k^*)^2 e^{-z} z^{b+1}}{b^2(b - 2\nu_k^*)} \Phi^2(\nu_k^* + 1, b + 1, z)
	\end{equation}\\
	\indent Substituting the bound on $I$ into (\ref{eq:truncation error}) results in
	
	\begin{equation}
		\label{eq:eps_N upper bound}
		\varepsilon_N \leq \frac{z^{b-1}e^{-z}}{2\Gamma(b)} \sum\limits_{k=N+1}^{\infty} \frac{b-2\nu_k^*}{(\nu_k^*)^2} n^{\nu_k^*}
	\end{equation}\\
	\noindent It is straightforward to show that $(b-2\nu_k^*) / (\nu_k^*)^2$ is not only positive (because $\nu_k^* < 0$), but that it also decreases as $\nu_k^*$ decreases. Therefore, we can move the coefficient on $n^{\nu_k^*}$ in (\ref{eq:eps_N upper bound}) outside of the sum by letting $k=N$, which leads to
	
	\begin{equation}
		\label{eq:simplified eps_N bound}
		\varepsilon_N \leq \frac{z^{b-1}e^{-z} (b-2\nu_N^*)}{2(\nu_N^*)^2\Gamma(b)} \sum\limits_{k=N+1}^{\infty} n^{\nu_k^*}
	\end{equation}\\
	\indent Let's now address the infinite series in (\ref{eq:simplified eps_N bound}). Given a lower bound $\Delta$ on $\nu_k^* - \nu_{k+1}^*$ for all $k>N$, we can write
	
	\begin{equation}
		\label{eq:sequence of sums}
		\sum\limits_{k=N+1}^{\infty} n^{\nu_k^*} < \sum\limits_{k=N+1}^{\infty} n^{\nu_N^* - \Delta(k - N)} = n^{\nu_N^* - \Delta} \sum\limits_{k=0}^{\infty} n^{-k\Delta}
	\end{equation}\\
	\noindent The series $\sum_{k=0}^{\infty} n^{-k\Delta}$ is a geometric series. Therefore,
	
	\begin{equation}
		\label{eq:geometric series}
		\sum\limits_{k=N+1}^{\infty} n^{\nu_k^*} < n^{\nu_N^* - \Delta} \left[\frac{n^{\Delta}}{n^{\Delta} - 1}\right] = \frac{n^{\nu_N^*}}{n^{\Delta} - 1}
	\end{equation}\\
	\noindent Substituting (\ref{eq:geometric series}) into (\ref{eq:simplified eps_N bound}) produces the upper bound in (\ref{eq:eps_N bound}).
\end{proof}

To summarize, the probability of false alarm is guaranteed to reside in the interval

\begin{equation}
	\label{eq:guaranteed range}
		\left[ P_{fa}^{(u)} - \frac{z^{b-1} e^{-z} n^{\nu_N^*} (b-2\nu_N^*)}{2 (\nu_N^*)^2 \Gamma(b) (n^{\Delta} - 1)}, P_{fa}^{(u)} \right]
\end{equation}\\
\noindent with

\begin{equation}
	\label{eq:Pfa upper bound}
	P_{fa}^{(u)} = \frac{\Gamma(b, z)}{\Gamma(b)} + \frac{z^b e^{-z}}{b\Gamma(b)} \Phi(1, b+1, z) - \frac{z^{2b} e^{-2z}}{b^2 \Gamma(b)} \sum\limits_{k=1}^N \frac{n^{\nu_k^*} \Phi^2(\nu_k^* + 1, b + 1, z)}{\int_0^z t^{b-1} e^{-t} \Phi^2(\nu_k^*,b,t) dt}
\end{equation}\\

\subsection{Numerical Generation of Probability Bounds}

This section provides an algorithm description for how to numerically generate the containment interval in (\ref{eq:guaranteed range}). First, there are four input/design parameters that need to be specified: the Wiener process dimension $m$ (which determines $b$), the length $n$ of the dimensionless time interval, a desired probability of false alarm, $P_{fa,des}$, and the number of terms $N$ to retain in the residue expansion. Next, (\ref{eq:approx Pfa}) is solved numerically to determine a threshold $y$ (and thus $z$) corresponding to $P_{fa,des}$, after which the $N$ (real and simple) zeros of $\Phi(\nu, b, z)$ closest to the origin are ascertained using a root finding algorithm.

At this point, the upper bound in (\ref{eq:Pfa upper bound}) can be computed. To get the lower bound in (\ref{eq:guaranteed range}), we need to determine $\Delta$ using the results from Theorems \ref{primary theorem} and \ref{secondary theorem}. The key is finding the pair of zeros $\nu_{k-1}^*$ and $\nu_k^*$ that satisfy the inequalities $\nu_k^* < \bar{\nu}^*$ and $z g_{k-1} < \beta_{k-1}$, where we remind the reader that $\beta_{k-1} = b - \nu_{k-1}^* - 1$. For the second inequality, let's substitute the expression for $g_{k-1}$ from (\ref{eq:z bar})

\begin{equation}
	\label{eq:second inequality}
	z \exp\left[ \frac{2\pi}{\sqrt{(2\beta_{k-1} + 1)(2\beta_{k-1} - 2b + 3)}}\right] < \beta_{k-1}
\end{equation}\\
\noindent The left-hand side of (\ref{eq:second inequality}) monotonically decreases with $\beta_{k-1}$ whereas the right-hand side is monotonically increasing, implying that there is one point $\hat{\beta}$ where both sides are equal. Thus, (\ref{eq:second inequality}) is satisfied for all $\beta_{k-1} > \hat{\beta}$. Or equivalently, with $\hat{\nu} = b - \hat{\beta}-1$, the inequality $z g_{k-1} < \beta_{k-1}$ is satisfied when $\nu_{k-1}^* < \hat{\nu}$.

Figure \ref{fig:zero landscape} shows an example of what the zero landscape might look like together with the critical values $\bar{\nu}^*$ and $\hat{\nu}$. In general, several zeros beyond $\nu_N^*$ need to be determined before finding the pair $\nu_{k-1}^*$ and $\nu_k^*$ that satisfies the requisite inequalities. When this pair has been found, Theorem \ref{primary theorem} enables determination of a lower bound $\Delta_{\infty}$ that bounds $\Delta_{k-1}$ and the spacing between all subsequent pairs of zeros $\Delta_k$, $\Delta_{k+1}$, $\cdots$. 

\begin{figure}[!htb]
	\centering
	\includegraphics[scale=0.62]{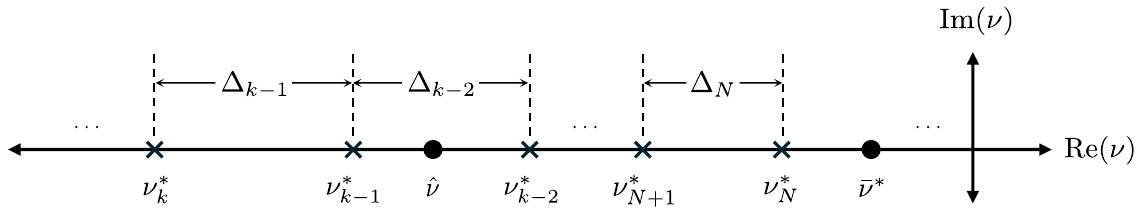}
	\caption{Example zero landscape.}
	\label{fig:zero landscape}
\end{figure}
\noindent Thus, the overall bound $\Delta$ needed in (\ref{eq:guaranteed range}) is $\Delta = \min(\Delta_N, \dots, \Delta_{k-2}, \Delta_{\infty})$. Algorithm \ref{alg:myalgorithm} summarizes the steps needed to compute $\Delta$.

\begin{algorithm}
	\caption{\vphantom{$a^{b^{b^b}}$}: Numerical Determination of $\Delta$ \vphantom{$a_{b_{b_b}}$}}\label{alg:myalgorithm}
	\algsetup{indent=2em}
	\begin{algorithmic}\vspace{1.5mm}
		\STATE \textbf{Input}: $P_{fa}$, $m$, $n$ and $N$ \vspace{1.5mm}
		\STATE Solve Eq. (\ref{eq:approx Pfa}) for the detection threshold $y$ \vspace{1.5mm}
		\STATE $b \gets m/2$ and $z \gets y^2/2$ \vspace{2mm}
		\STATE Solve $z \hspace{0.2mm} e^{\hspace{0.2mm}2\pi / \sqrt{(2\hat{\beta} + 1)(2\hat{\beta} - 2b + 3)}} - \hat{\beta}$ = 0 for $\hat{\beta}$ \vspace{1.5mm}
		\STATE $\hat{\nu} \gets b - \hat{\beta} - 1$ \vspace{1.5mm}
		\STATE $d_0 \gets (b-1)^2 (2b-3) \quad , \quad d_1 \gets 0 \quad , \quad d_2 \gets b^2 - 2$ \vspace{1.5mm}
		\STATE $d_3 \gets -(b-1)^2/\pi \hspace{9mm} , \quad d_4 \gets 1 \quad , \quad d_5 \gets -1/\pi$ \vspace{1.5mm}
		\STATE $d_6 \gets 1/(4\pi^2)$ \vspace{1.5mm}
		\STATE Find all roots $u_i$ of the polynomial $p(u) = \sum_{l=0}^6 d_l u^l$ \vspace{1.5mm}
		\STATE $\bar{\nu}^* \gets 0$ \vspace{1mm}
		\FOR{$i=1$ to $6$} \vspace{2.5mm}
		\STATE $c \gets \displaystyle \frac{b}{2} - \frac{1}{2}\sqrt{(b-1)^2 + u_i^2}$ \vspace{2.0mm}
		\STATE $\text{numer} \gets 4\pi (b - c - 1) (b - 2 c)$ \vspace{1.5mm}
		\STATE $\text{denom} \gets \left[(b - 2c)^2 - (b-1)^2 \right]^{3/2}$ \vspace{1.5mm}
		\IF{$\text{numer} /\text{denom} = 1$ \AND $\operatorname{Im}(c) = 0$} \vspace{1.5mm}
		\STATE $\bar{\nu}^* \gets \operatorname{min}(\bar{\nu}^*, c)$ \vspace{1.5mm}
		\ENDIF \vspace{1mm}
		\ENDFOR\vspace{1mm}
		\STATE Let $\nu_1^*$ and $\nu_2^*<\nu_1^*$ be the two zeros of $\Phi(\nu, b, z)$ closest to the origin. Obtain $\nu_1^*$
		\STATE and $\nu_2^*$ using a numerical solver. \vspace{1.5mm}
		\STATE $\nu^*[1] \gets \nu_1^*$ \AND $\nu^*[2] \gets \nu_2^*$ \AND $\Delta[1] \gets \nu_1^* - \nu_2^*$  \vspace{1.5mm}
		\STATE $i \gets 2$ \vspace{1mm}
		\WHILE{$\nu^*[i-1] \geq \hat{\nu}$ \OR $\nu^*[i] \geq \bar{\nu}^*$ \OR $i<N+1$}  \vspace{1mm}
		\STATE Numerically determine the next zero $\nu_{i+1}^* < \nu_i^*$ of $\Phi(\nu,b,z)$ adjacent to $\nu_i^*$ \vspace{1.5mm}
		\STATE $\nu^*[i+1] \gets \nu_{i+1}^*$ \AND $\Delta[i] \gets \nu_i^* - \nu_{i+1}^*$ \vspace{1.5mm}
		\STATE $i \gets i+1$ \vspace{1mm}
		\ENDWHILE \vspace{0.5mm}
		\STATE $k \gets \text{length}(\nu^*)$ \vspace{1mm}
		\STATE $\beta \gets b - 1 - \nu^*[k]$ \vspace{2.5mm}
		\STATE $g \gets e^{2\pi / \sqrt{(b - 2 \nu^*[k]) ^ 2 - (b - 1) ^ 2}}$ \vspace{1.5mm}
		\STATE $\Delta_{\text{inf}} \gets \beta - \displaystyle \frac{\beta}{4g}\left[ 2 + \sqrt{z/\beta} (g - 1) \right]^2$ \vspace{2.5mm}
		\STATE $\Delta \gets \operatorname{min}(\Delta[N], \hdots, \Delta[k-2], \Delta_{\text{inf}})$
	\end{algorithmic}
\end{algorithm}

\subsection{Results}

We are now positioned to explore the accuracy of (\ref{eq:approx Pfa}). Specifically, we seek to determine how closely the true probability of false alarm agrees with the expected value. To this aim, let's focus on the desired value of $P_{fa} = 10^{-4}$ for Wiener process dimensions $m = 1, 3, 7$ and $10$, and non-dimensional time intervals of $n = 5, 10, 30$ and $100$. In addition, we will let $N=3$ in (\ref{eq:guaranteed range}) and (\ref{eq:Pfa upper bound}). A detailed analysis is given first for $m=3$ and $n=10$.

Following the first half of Algorithm \ref{alg:myalgorithm}, we determine that $y = 5.308$, $\hat{\nu} = -16.417$ and $\bar{\nu}^* = -2.153$. The next step is to obtain the sequence of zeros $\nu_1^*, \hdots, \nu_k^*$ such that $\nu_{k-1}^* < \hat{\nu}$ and $\nu_k^* < \bar{\nu}^*$, the results of which are summarized in Table \ref{tab:zeros of Phi}. Notice that we needed to determine the first eleven zeros until the necessary inequalities are satisfied.

\begin{table}[h]
	\renewcommand{\arraystretch}{1.5}
	\caption{Zeros of $\Phi(\nu, m/2, y^2/2)$ for $m=3$ and $y = 5.308$}\label{tab:zeros of Phi}%
	\begin{tabular*}{230pt}{@{\extracolsep\fill}cccc@{}}
		\toprule
		Zero & Location & Zero & Location \\
		\midrule
		$\nu_1^*$     & $-4.014\text{E}\!-\!05$ & $\nu_7^*$    & $-9.035$ \\
        $\nu_2^*$     & $-1.003$                & $\nu_8^*$    & $-11.655$ \\
        $\nu_3^*$     & $-2.054$                & $\nu_9^*$    & $-14.628$ \\
        $\nu_4^*$     & $-3.296$                & $\nu_{10}^*$ & $-17.953$ \\
        $\nu_5^*$     & $-4.855$                & $\nu_{11}^*$ & $-21.629$ \\
        $\nu_6^*$     & $-6.767$                &              &    \\
		\botrule
	\end{tabular*}
\end{table}
\noindent The last step is to compute the bound $\Delta$, which one can verify is $\Delta = 0.516$. Substituting into (\ref{eq:guaranteed range}), we conclude that the true probability of false alarm resides in the interval

\begin{equation*}
	P_{fa,\text{true}} \in [9.99199 ,9.99282] \times 10^{-5}
\end{equation*}\\
\noindent Observe that with just three residues ($N=3$), we are able to place $P_{fa,\text{true}}$ within a tight interval. It is also comforting to see that the approximation in (\ref{eq:approx Pfa}) is quite accurate, yielding a detection threshold that produces a true false alarm probability within $0.08\%$ of the desired value of $10^{-4}$.

\begin{table}[h]
	\renewcommand{\arraystretch}{1.4}
	\caption{Maximum percent difference between true and desired probability of false alarm}\label{tab:percent differences}
	\begin{tabular*}{333pt}{@{\extracolsep\fill}ccccc@{}}
		\toprule%
		& \multicolumn{4}{@{}c@{}}{Length of Non-dimensional Time Interval, $n$} \\\cmidrule{2-5}%
		Process Dimension, $m$ & $n=5$ & $n=10$ & $n=30$ & $n=100$ \\
		\midrule
		$1$  & $0.34$ & $0.08$ & $0.12$ & $0.22$ \\
		$3$  & $0.37$ & $0.08$ & $0.39$ & $0.54$ \\
		$7$  & $0.43$ & $0.22$ & $0.65$ & $0.87$ \\
		$10$  & $0.47$ & $0.28$ & $0.78$ & $1.02$ \\
		\botrule
	\end{tabular*}
\end{table}
Similar results are obtained for other combinations of $m$ and $n$ that are summarized in Table \ref{tab:percent differences}. The largest percent difference observed is $1\%$, which occurs when monitoring a ten-dimensional Wiener process over the time interval $[1, 100]$. This level of performance is satisfactory for most applications. If this is not the case, the discrepancy between the true and desired probability of false alarm can be reduced by iterating on the threshold $y$ until the percent difference reaches an acceptable level.

\section{Conclusion}\label{sec:conclusions}

A lower bound on the separation between consecutive zeros of $\Phi(a,b,z)$ was derived for variable $a$ and $(b,z) \in \mathbb{R}^+$ known and fixed. Conditions for monotonicity of the bound were derived and used to analyze the accuracy of asymptotic approximations for the first passage probability of an $m$-dimensional Wiener process. We showed that when such approximations are used, the true probability is within $1\%$ of the expected value over a range of process dimensions and observation intervals. The validity of a residue expansion for the first passage probability was also rigorously proven using recent results from value distribution theory. One direction for future research is to obtain an improved integral bound over that given in Lemma \ref{lem:integral_bound} that is valid for all $a \in \mathbb{R}^-$, which would eliminate the constraint $z<-1-a+b$. Another avenue to explore is whether the results of this paper can be used to infer properties of other special functions, many of which can be written in terms of the confluent hypergeometric function. 

\backmatter

\begin{appendices}
\renewcommand\theHequation{AABB\arabic{equation}}

\section{Integral Derivations}\label{Appendix A}

This appendix derives (\ref{eq:integral 1}) and (\ref{eq:integral 2}). Recall from (\ref{eq:Whittaker M function}) that the Whittaker $\mathcal{M}$ function is defined as

\begin{equation}
	\label{eq:copy Whittaker M}
	\mathcal{M}_{\varkappa, \mu/2}(z) = \frac{1}{\Gamma(1+\mu)} z^{(1+\mu)/2}e^{-z/2} \Phi\left(\frac{1+\mu}{2} - \varkappa, 1 + \mu, z\right)
\end{equation}\\
\noindent Equation (4a) in \cite[p. 113]{Buchholz} gives the following indefinite integral for $\varkappa \neq \lambda$

\begin{equation}
	\label{eq:4a Buchholz}
	\begin{split}
		(\varkappa - \lambda) &\int \mathcal{M}_{\varkappa, \mu/2}(z) \mathcal{M}_{\lambda,\mu/2}(x)\frac{dz}{z} = \\ & \mathcal{M}_{\varkappa,\mu/2}(z) \mathcal{M}^{\prime}_{\lambda,\mu/2}(z) - \mathcal{M}^{\prime}_{\varkappa,\mu/2}(z)\mathcal{M}_{\lambda,\mu/2}(z)
	\end{split}
\end{equation}\\
\noindent such that $\mathcal{M}^{\prime}_{\varkappa,\mu/2}(z)$ is the derivative of $\mathcal{M}_{\varkappa,\mu/2}(z)$ with respect to $z$. From (\ref{eq:copy Whittaker M}),

\begin{equation}
	\label{eq:dM_dz}
	\mathcal{M}^{\prime}_{\varkappa, \mu/2}(z) =  \frac{e^{-z/2}z^{b/2}}{\Gamma(b)} \left[ \frac{\xi}{b} \Phi(\xi+1,b+1,z) + \Phi(\xi, b, z) \left( \frac{b}{2z} - \frac{1}{2} \right) \right]
\end{equation}\\
\noindent with $\xi = (1+\mu)/2 - \varkappa$ and $b = 1+\mu$. After substituting (\ref{eq:copy Whittaker M}) and (\ref{eq:dM_dz}) into (\ref{eq:4a Buchholz}) and defining $\eta = (1+\mu)/2 - \lambda$, we get the following expression for $\xi \neq \eta$

\begin{equation}
	\label{eq:int product of phis}
	\begin{split}
		\int z^{b-1}e^{-z} & \Phi(\xi, b, z) \Phi(\eta, b, z) dz = \\ & \frac{e^{-z} z^b}{b(\eta - \xi)} \left[ \eta \Phi(\xi, b, z) \Phi(\eta + 1, b + 1, z) - \xi \Phi(\eta, b, z) \Phi(\xi + 1, b + 1, z) \right]
	\end{split}
\end{equation}\\
\indent Next, we convert (\ref{eq:int product of phis}) to a definite integral. Writing (\ref{eq:int product of phis}) generically as $\int g(z)dz = G(z)$, we seek a point $c$ where $G(c) = 0$, in which case $\int_c^z g(t) dt = G(z)$. To this aim, consider the behavior of $G(z)$ near $z=0$. Entry 13.5.5 in \cite{Abramowitz} shows that $\Phi(a,b,z) \rightarrow 1$ as $\lvert z \rvert \rightarrow 0$, provided that $b$ is not a negative integer. Substituting this result into the right-hand side of (\ref{eq:int product of phis}), it is straightforward to show that as $\lvert z \rvert \rightarrow 0$, $G(z) \rightarrow z^b/b$. Thus, provided that $\text{Re}(b) > 0$, $G(z) \rightarrow 0$ as $\lvert z \rvert \rightarrow 0$, and (\ref{eq:int product of phis}) can be written as the definite integral

\begin{equation}
	\begin{split}
		\int\limits_0^z t^{b-1}e^{-t} & \Phi(\xi, b, t) \Phi(\eta, b, t) dt = \\ & \frac{e^{-z} z^b}{b(\eta - \xi)} \left[\eta \Phi(\xi, b, z) \Phi(\eta + 1, b + 1, z) - \xi \Phi(\eta, b, z) \Phi(\xi + 1, b + 1, z) \right]
	\end{split}
\end{equation}\\
\indent A similar approach is used to derive (\ref{eq:integral 2}), starting from Eq. (4$\beta$) in \cite[p. 114]{Buchholz}. There is an error in \cite{Buchholz} that is corrected in Appendix \ref{Appendix C}, leading to the relation\footnote[7]{In \cite{Buchholz}, the derivative on the right-hand side of (\ref{eq:4 beta Buchholz}) is written $d/dx$, and not $d/d(cx)$, which produces an additional factor $c$ that should not exist.}
\begin{equation}
	\label{eq:4 beta Buchholz}
	\begin{split}
		\int \left( \frac{1}{2} - \frac{\varkappa}{cz} \right) & \mathcal{M}_{\varkappa, \mu/2}^2 (cz) d(cz) = -cz \mathcal{M}_{\varkappa, \mu/2}^2(cz) \left[-\frac{1}{4} + \frac{\varkappa}{cz} + \frac{1 - \mu^2}{4c^2 z^2} \right] \\ &-cz [\mathcal{M}_{\varkappa, \mu/2}^{\prime}(cz)]^2 + \frac{1}{2} \frac{d}{d(cz)} \mathcal{M}_{\varkappa, \mu/2}^2(cz)
	\end{split}
\end{equation}\\
\noindent Let $c=1$ and recall from earlier that as $\lvert z \rvert \rightarrow 0$, $\Phi(a, b, z) \rightarrow 1$ for $b$ not equal to a negative integer. Then from (\ref{eq:copy Whittaker M}) and (\ref{eq:dM_dz}), we have the following limiting behavior as $\lvert z \rvert \rightarrow 0$
\begin{equation}
	\label{eq:M and Mprime limits}
	\begin{split}
		\mathcal{M}_{\varkappa, \mu/2}(z) & \approx z^{b/2} / \Gamma(b) \\[1ex] \mathcal{M}_{\varkappa, \mu/2}^{\prime}(z) & \approx \frac{z^{b/2}}{\Gamma(b)} \left[ \frac{a}{b} + \frac{b}{2z} - \frac{1}{2} \right]
	\end{split}
\end{equation}\\
\noindent where $a = (1+\mu)/2 - \varkappa$ and $b = 1+\mu$.

Denote the right-hand side of (\ref{eq:4 beta Buchholz}) as $G(z)$. After substituting the expressions in (\ref{eq:M and Mprime limits}), we can see that
\begin{equation}
	\text{As } \lvert z \rvert \rightarrow 0 \, , \, G(z) \rightarrow \frac{z^b [a z (b-a) - b \varkappa]}{b^2}
\end{equation}\\
\noindent If $\text{Re}(b) > 0$, $G(z) \rightarrow 0$ as $\lvert z \rvert \rightarrow 0$, and (\ref{eq:4 beta Buchholz}) can be written as the definite integral
\begin{equation}
	\begin{split}
		\int\limits_0^z \left(\frac{1}{2} - \frac{\varkappa}{t} \right) & \mathcal{M}_{\varkappa, \mu/2}^2 (t) dt = -z \mathcal{M}_{\varkappa, \mu/2}^2(z) \left[-\frac{1}{4} + \frac{\varkappa}{z} + \frac{1 - \mu^2}{4z^2} \right] \\ &-z [\mathcal{M}_{\varkappa, \mu/2}^{\prime}(z)]^2 + \mathcal{M}_{\varkappa,\mu/2}(z) \mathcal{M}_{\varkappa,\mu/2}^{\prime}(z)
	\end{split}
\end{equation}\\
\noindent Substituting the definitions for $\mathcal{M}_{\varkappa, \mu/2}(z)$ and $\mathcal{M}_{\varkappa, \mu/2}^{\prime}(z)$ and defining $k=b/2-a$, we arrive at the desired result in (\ref{eq:integral 2}).

\section{Inverse Laplace Transform as a Residue Expansion}\label{Appendix B}

This appendix is concerned with the inverse Laplace transform of the function

\begin{equation}
	\label{eq:ratio of hypergeometrics}
	G(\nu) = \frac{\Phi(\nu+1, b+1, z)}{\nu \Phi(\nu, b, z)} = \frac{b}{\nu^2} \frac{\Phi^{\prime}(\nu, b, z)}{\Phi(\nu, b, z)}
\end{equation}\\
\noindent where $(b, z) \in \mathbb{R}^+$ have known fixed values and $\Phi^{\prime}(\nu, b, z)$ is the derivative of $\Phi(\nu, b, z)$ with respect to $z$. From Proposition \ref{prop:zeros a_star of Phi}, $G(\nu)$ has a simple pole at $\nu=0$ and an infinite set of simple poles on the $-\mathrm{Re}(\nu)$ axis corresponding to the zeros of $\Phi(\nu, b, z)$. The set of zeros of $\Phi(\nu, b, z)$ must be infinite, since otherwise we would infer asymptotic behavior inconsistent with the behavior of $\Phi(\nu, b, z)$ for $\nu \rightarrow \infty$ (see \cite[p. 185]{Buchholz}).

The inverse Laplace transform of $G(\nu)$ is defined by the complex line integral \cite{Schiff}

\begin{equation}
	\label{eq:line integral}
	\mathcal{L}^{-1}\{G(\nu)\} = \lim_{\gamma \rightarrow \infty} \frac{1}{2\pi i} \int\limits_{\varepsilon - i \gamma}^{\varepsilon + i \gamma} e^{\nu u} G(\nu) d\nu
\end{equation}\\
\noindent such that $\varepsilon > 0$ is an arbitrarily small number\footnote[8]{In general, $\varepsilon$ must be greater than the real part of all poles of $G(\nu)$. For us, the poles all happen to be in the left-half plane, so that we can take $\varepsilon$ to be an arbitrarily small number.}. When $G(\nu)$ has an infinite number of poles, (\ref{eq:line integral}) is usually evaluated by examining the limiting behavior of the integral around the semi-circular contour shown in Fig. \ref{fig:closed contour} as $R \rightarrow \infty$. For finite $R$, the contour encloses a finite set of simple poles, so that from Cauchy's residue theorem,

\begin{figure}[!htb]
	\centering
	\includegraphics[scale=0.75]{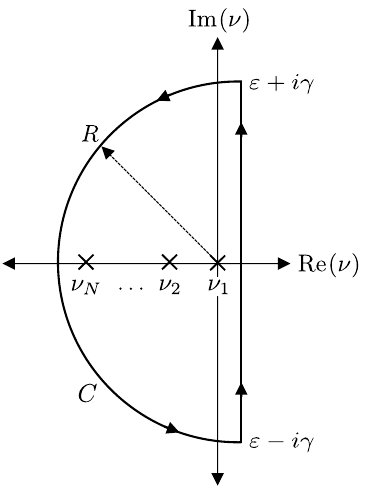}
	\caption{Closed contour used to evaluate $\mathcal{L}^{-1}\{ G(\nu)\}$. The $\times$'s indicate the poles of $G(\nu)$.}
	\label{fig:closed contour}
\end{figure}

\begin{equation}
	\label{eq:Cauchy theorem}
	\frac{1}{2\pi i} \int\limits_{\varepsilon - i \gamma}^{\varepsilon + i \gamma} e^{\nu u} G(\nu) d\nu + \frac{1}{2\pi i} \int\limits_C e^{\nu u} G(\nu) d\nu = \sum_{k=1}^N \text{Res}(e^{\nu u} G, \nu_k^*) \text{ with } \lvert \nu_k^* \rvert < R
\end{equation}\\
\noindent where $\text{Res} (e^{\nu u} G, \nu_k^*)$ is the residue of $e^{\nu u} G(\nu)$ at the pole $\nu_k^*$.

If we can show that $\int_C e^{\nu u} G(\nu) d\nu \rightarrow 0$ as $R \rightarrow \infty$, then $\mathcal{L}^{-1}\{G(\nu)\}$ reduces to an infinite residue expansion, i.e.,

\begin{equation}
	\label{eq:Residue expansion}
	\lim_{\gamma \rightarrow \infty} \frac{1}{2\pi i} \int\limits_{\varepsilon - i \gamma}^{\varepsilon + i \gamma} e^{\nu u} G(\nu) d\nu = \mathcal{L}^{-1} \{G(\nu)\} = \sum_{k=1}^{\infty} \text{Res}(e^{\nu u} G, \nu_k^*)
\end{equation}\\
\noindent Jordan's lemma states that for $\theta \in [\pi/2, 3\pi/2]$, if $\lvert G(Re^{i\theta}) \rvert \leq q(R)$, with $q(R) \rightarrow 0$ uniformly as $R\rightarrow \infty$, then $\int_C e^{\nu u} G(\nu) d\nu \rightarrow 0$ as $R \rightarrow \infty$ \cite{Schiff}. To ascertain whether these conditions are met, we first determine a growth restriction for $G(\nu)$.

\begin{theorem}
	Let $G(\nu)$ be as defined in (\ref{eq:ratio of hypergeometrics}) and let $\nu$ be any point on $C$ in Fig. \ref{fig:closed contour} such that $R>>1$ and $\Phi(-R, b, z) \neq 0$. Then with $(x)_i = x (x+1) \cdots (x + i - 1)$, there exist finite $M_1$ and $M_2$ such that
	
	\begin{equation}
		\label{eq:growth restriction}
		\lvert G(\nu) \rvert \leq \left\{\begin{array}{lr}
		\displaystyle\frac{M_1}{R^2}, & \text{for } \xi \leq 1\\[3ex]
		\displaystyle \frac{M_2\ln R}{R^2}, & \text{for } \xi > 1
	\end{array}\right.
	\end{equation}\\
	\noindent where $\xi = \displaystyle \max_{1\leq i \leq j} \left| (\nu)_i / (b)_i \right|$ and $j \in \mathbb{Z}^+$ is the smallest integer such that $\text{Re}(\nu) + j > 0$ and $\lvert \nu + j \rvert < \lvert b+j \rvert$.
\end{theorem}

\begin{proof}
	First recognize $\Phi^{\prime}(\nu, b, z) / \Phi(\nu, b, z)$ as the logarithmic derivative of $\Phi(\nu, b, z)$, a quantity that has been studied extensively under Nevanlinna's value distribution theory. In particular, let $f(z)$ be a meromorphic function satisfying $f(0) = 1$ with a set of zeros $\{a_m\}$ and a set of poles $\{b_n\}$. Inside the disk $\lvert z \rvert < s$, $f^{\prime}(z)/f(z)$ satisfies the bound \cite[eq. $(1.3^{\prime})$, p. 88]{Goldberg}
	
	\begin{equation}
		\label{eq:logarithmic derivative bound}
		\left| \frac{f^{\prime}(z)}{f(z)} \right| \leq \frac{4sT(s,f)}{(s - \lvert z \rvert)^2} + 2 \sum_{\lvert c_q \rvert < s} \frac{1}{\lvert z - c_q \rvert} \;\; , \;\; \lvert z \rvert < s
	\end{equation}\\	
	\noindent where $\{c_q\}$ is the set-theoretic sum of $\{a_m\}$ and $\{b_n\}$ and $T(s,f)$ is the Nevanlinna characteristic.
	
	For a nonconstant meromorphic function $f(z)$, $T(s,f) = m(s,f) + N(s,f)$, where $m(s,f)$ and $N(s,f)$ are the proximity and counting functions, respectively \cite[eq. (12)]{Luo}. The counting function is defined as
	
	\begin{equation}
		\label{eq:counting function}
		N(s,f) = \int\limits_{0}^{s} \frac{n(t,f) - n(0,f)}{t} dt + n(0, f) \ln s
	\end{equation}\\	
	\noindent such that $n(t,f)$ is the number of poles of $f$ in the closed disc $\overline{D(0,t)} = \{z : \lvert z \rvert \leq t \}$, counting multiplicities. For entire functions, which have no poles, $n(0,f) = n(t,f) = 0$, implying that $N(s,f) = 0$ and therefore $T(s,f) = m(s,f)$. Thus, for the entire function $\Phi(\nu, b, z)$, (\ref{eq:logarithmic derivative bound}) can be written as
	
	\begin{equation}
		\label{eq:log der bound}
		\left| \frac{\Phi^{\prime}(\nu, b, z)}{\Phi(\nu, b, z)} \right| \leq \frac{4sm[s,\Phi(\nu, b, z)]}{(s - \lvert z \rvert)^2} + 2 \sum_{\lvert c_q \rvert < s} \frac{1}{\lvert z - c_q \rvert} \;\; , \;\; \lvert z \rvert < s
	\end{equation}\\	
	\indent We are interested in analyzing the bound in (\ref{eq:log der bound}) along the circular arc in Fig. \ref{fig:closed contour}. That is, when $\nu = R e^{i \theta}$ and $\theta \in [\pi/2 - \delta, 3\pi / 2 + \delta]$, with $\delta = \sin^{-1}(\varepsilon / R)$. The function $\Phi(\nu, b, z)$ is analytic over the entire complex plane and therefore can only have a finite number of zeros in the disk $\lvert z \rvert < s$, provided that $s$ is finite\footnote[9]{We will only be concerned with finite $z$, so that a finite $s$ can always be found that satisfies $\lvert z \rvert < s$.} \cite[Th. 6.39]{Willms}. We also know from Proposition \ref{prop:zeros a_star of Phi} that when $b \in \mathbb{R}^+$ and $\nu$ is complex (i.e., $\theta \neq \pi$), all elements of $\{c_q\}$ are complex, in which case $\lvert z - c_q \rvert$ is guaranteed to be finite for real $z$. However, when $\nu = -R$, there will be certain values of $R$ where $\Phi(-R, b, z) = 0$, implying that $z \in \{c_q\}$. We do not need to consider this possibility because the zeros $\nu^*$ of $\Phi(\nu,b,z)$ are isolated, meaning that $R$ can always be increased so that $z \notin \{c_q\}$ and $\lvert z-c_q \rvert$ is finite even when $\theta = \pi$. These facts lead to the conclusion that for $\lvert z \rvert < s < \infty$,
	
	\begin{equation}
		\label{eq:bounded sum}
		2 \sum_{\lvert c_q \rvert < s} \frac{1}{\lvert z - c_q \rvert} = \mathcal{O}(1) \; \; , \; \; R >> 1 \; \text{and} \; \Phi(-R, b, z) \neq 0
	\end{equation}\\	
	\indent Now let's turn our attention to the other term in (\ref{eq:log der bound}). Clearly, $\text{Re}(\nu) \leq \varepsilon$ on $C$. Furthermore, since $b\in \mathbb{R}^+$ and $\varepsilon$ can be made arbitrarily small, we have $\text{Re}(\nu) < \text{Re}(b)$. In this case, \cite{Luo} shows that
	
	\begin{equation}
		\label{eq:proximity bound}
		m[s, \Phi(\nu, b, z)] \leq s + \ln \left[1+ \max_{1\leq i \leq j} \left| \frac{(\nu)_i}{(b)_i} \right| \right]
	\end{equation}\\
	\noindent where $(\nu)_i = \nu (\nu + 1) \cdots (\nu + i - 1)$ and $j \in \mathbb{Z}^+$ is the smallest integer such that $\text{Re}(\nu) + j > 0$ and $\lvert \nu + j \rvert < \lvert b+j \rvert$. Assume $\lvert (\nu)_i / (b)_i \rvert$ is maximized when $i=k$ so that
	
	\begin{equation}
		\label{eq:simplified proximity bound}
		m[s, \Phi(\nu, b, z)] \leq s + \ln \left[1+ \left| \frac{(\nu)_k}{(b)_k} \right| \right]
	\end{equation}\\	
	\indent If $\lvert (\nu)_k / (b)_k \rvert \leq 1$, we have $\ln [1 + \lvert (\nu)_k / (b)_k \rvert] = \mathcal{O} (1)$, and upon substituting into (\ref{eq:log der bound}), we get
	
	\begin{equation}
		\label{eq:case z < 1}
		\frac{4sm[s, \Phi(\nu, b, z)]}{(s - \lvert z \rvert)^2} \leq \frac{4s(s + \mathcal{O}(1))}{(s - \lvert z \rvert)^2} = \mathcal{O}(1) \,\, , \,\, \lvert z \rvert < s
	\end{equation}\\	
	\noindent Now suppose that $\lvert (\nu)_k / (b)_k \rvert > 1$. First observe that
	
	\begin{equation}
		\label{eq:natural log bound}
		\ln \left[1 + \left| \frac{(\nu)_k}{(b)_k} \right|     \right] \leq \ln \left[2 \frac{\lvert(\nu)_k \rvert}{\lvert(b)_k \rvert} \right]
	\end{equation}\\	
	\noindent Also observe that $(\lvert \nu \rvert)_k \geq \lvert (\nu)_k \rvert$ so that we can write
	
	\begin{equation}
		\label{eq:looser bound}
		\ln \left[2 \frac{\lvert(\nu)_k \rvert}{\lvert(b)_k \rvert} \right] \leq \ln 2 + \ln [(\lvert \nu \rvert)_k] - \ln [ \lvert (b)_k \rvert]
	\end{equation}\\	
	\noindent Now let's substitute the definition $(\lvert \nu \rvert)_k = \Gamma(\lvert \nu \rvert + k) / \Gamma(\lvert \nu \rvert)$ so that $\ln [(\lvert \nu \rvert)_k] = \ln \Gamma(\lvert \nu \rvert + k) - \ln \Gamma(\lvert \nu  \rvert)$. Note that $\lvert \nu \rvert$ and $k$ are both positive real numbers so that $\text{arg}(\lvert \nu \rvert + k) = \text{arg}(\lvert \nu \rvert) = 0$. In this case, \cite{Whittaker} shows that as $\lvert \nu \rvert \rightarrow \infty$,
	
	\begin{equation}
		\label{eq:log-gamma asymptotics}
		\ln \Gamma (\lvert \nu \rvert + k) = \left(\lvert \nu \rvert + k - \frac{1}{2} \right) \ln (\lvert \nu \rvert) - \lvert \nu \rvert + \frac{1}{2} \ln (2\pi) + o(1)
	\end{equation}\\	
	\noindent with the term $o(1)$ going to zero as $\lvert \nu \rvert \rightarrow \infty$. Thus, for large $\lvert \nu \rvert$, $\ln[(\lvert \nu \rvert)_k] \approx k \ln (\lvert \nu \rvert) + o(1)$. This allows us to conclude that for $\lvert (\nu)_k / (b)_k \rvert > 1$ and $\lvert \nu \rvert >>1$,
	
	\begin{equation}
		\label{eq:proximity function bound}
		m[s, \Phi(\nu,b,z)] \leq s + \ln \left[1 + \left| \frac{(\nu)_k}{(b)_k} \right| \right] \leq s + k \ln \lvert \nu \rvert + \ln \left[ \frac{2}{\lvert (b)_k \rvert} \right] + o(1)
	\end{equation}\\	
	\indent Substituting (\ref{eq:bounded sum}), (\ref{eq:case z < 1}) and (\ref{eq:proximity function bound}) into (\ref{eq:log der bound}) and defining $\xi = \lvert (\nu)_k / (b)_k \rvert$, we can see that for some finite $M_1$ and $M_2$, $G(\nu)$ satisfies the following growth restriction on $C$
	
	\begin{equation}
		\label{eq:decay of G(nu)}
		\lvert G(\nu) \rvert = \frac{1}{\lvert \nu \rvert^2} \left| \frac{\Phi^{\prime}(\nu, b, z)}{\Phi(\nu, b, z)} \right| \leq \left\{\begin{array}{lr}
			\displaystyle \frac{M_1}{R^2}, & \text{for } \xi \leq 1\\[3ex]
			\displaystyle \frac{M_2 \ln R}{R^2}, & \text{for } \xi > 1
		\end{array}\right.
	\end{equation}
\end{proof}

Both bounds in (\ref{eq:decay of G(nu)}) converge uniformly to $0$ as $R\rightarrow \infty$. Thus, $G(\nu)$ satisfies the conditions of Jordan's lemma and we can conclude that the residue expansion in (\ref{eq:Residue expansion}) is a valid expression for $\mathcal{L}^{-1}\{G(\nu)\}$.

\section{Correction to Formula $4\beta$ in \cite[p. 114]{Buchholz}}\label{Appendix C}

This appendix derives (\ref{eq:4 beta Buchholz}), starting from Eq. (4b) in \cite[p. 113]{Buchholz}, which states

\begin{equation}
	\label{eq:4b}
	\begin{split}
		(c_1 - c_2) & \int \left(\frac{c_1 + c_2}{4} - \frac{\varkappa}{t}\right) \mathcal{M}_{\varkappa, \mu/2}(c_1 t) \mathcal{M}_{\varkappa, \mu/2}(c_2 t) dt \\[1ex] & = -c_2 \mathcal{M}_{\varkappa, \mu/2}(c_1 t) \mathcal{M}_{\varkappa, \mu/2}^{\prime}(c_2 t) + c_1 \mathcal{M}_{\varkappa, \mu/2}^{\prime}(c_1 t) \mathcal{M}_{\varkappa, \mu/2}(c_2 t)
	\end{split}
\end{equation}\\
\noindent where

\begin{equation}
	\label{eq:notation}
	\mathcal{M}_{\varkappa, \mu/2}^{\prime}(c_2 t) = \left. \frac{d\mathcal{M}_{\varkappa,\mu/2}(x)}{dz} \right\rvert_{z=c_2 t}
\end{equation}\\
\noindent We seek to evaluate (\ref{eq:4b}) as $c_1 \rightarrow c_2$. To this aim, let $c_1 = c_2 + \varepsilon$ and define

\begin{equation}
	I = \int \left(\frac{c_1 + c_2}{4} - \frac{\varkappa}{t}\right) \mathcal{M}_{\varkappa, \mu/2}(c_1 t) \mathcal{M}_{\varkappa, \mu/2}(c_2 t) dt
\end{equation}\\
\noindent so that with $c_2 = c$, (\ref{eq:4b}) can be written as

\begin{equation}
	\label{eq:simplified 4b}
	I = -\frac{c}{\varepsilon}\mathcal{M}_{\varkappa, \mu/2}[(c+\varepsilon)t] \mathcal{M}_{\varkappa,\mu/2}^{\prime}(ct) + \frac{c+\varepsilon}{\varepsilon} \mathcal{M}_{\varkappa,\mu/2}^{\prime}[(c+\varepsilon)t] \mathcal{M}_{\varkappa,\mu/2}(ct)
\end{equation}\\
\indent Expanding $\mathcal{M}_{\varkappa, \mu/2}[(c+\varepsilon)t]$ and $\mathcal{M}_{\varkappa,\mu/2}^{\prime}[(c+\varepsilon)t]$ to first order about $\varepsilon = 0$ yields

\begin{equation}
	\label{eq:Maclaurin series}
	\begin{split}
		I & = -\frac{c}{\varepsilon}[\mathcal{M}_{\varkappa,\mu/2}(ct) + t \varepsilon \mathcal{M}_{\varkappa,\mu/2}^{\prime}(ct)] \mathcal{M}_{\varkappa,\mu/2}^{\prime}(ct) \\[1ex] & + \frac{c+\varepsilon}{\varepsilon} [\mathcal{M}_{\varkappa,\mu/2}^{\prime}(ct) + t \varepsilon \mathcal{M}_{\varkappa,\mu/2}^{\prime\prime}(ct)] \mathcal{M}_{\varkappa,\mu/2}(ct)
	\end{split}
\end{equation}\\
\noindent After distributing and taking the limit $\varepsilon \rightarrow 0$, (\ref{eq:Maclaurin series}) simplifies to

\begin{equation}
	\label{eq:simplified series}
	\begin{split}
		I & = \int \left(\frac{c}{2} - \frac{\varkappa}{t}\right) \mathcal{M}_{\varkappa,\mu/2}^2(ct)dt \\[1ex] & = -ct [\mathcal{M}_{\varkappa,\mu/2}^{\prime}(ct)]^2 + \mathcal{M}_{\varkappa,\mu/2}(ct) \mathcal{M}_{\varkappa,\mu/2}^{\prime}(ct) + ct \mathcal{M}_{\varkappa,\mu/2}^{\prime\prime}(ct) \mathcal{M}_{\varkappa,\mu/2}(ct)
	\end{split}
\end{equation}\\
\noindent From the definition of $\mathcal{M}_{\varkappa,\mu/2}(z)$ in (\ref{eq:Whittaker M function}) and with

\begin{equation}
	\label{eq:parameter defs}
	a = \frac{1+\mu}{2} - \varkappa, \quad b = 1 + \mu, \quad \gamma = \frac{z^{b/2} e^{-z/2}}{\Gamma(b)}
\end{equation}\\
\noindent the derivatives of $\mathcal{M}_{\varkappa,\mu/2}(z)$ are given by

\begin{equation}
	\label{eq:M prime and dprime}
	\begin{split}
		&\mathcal{M}_{\varkappa,\mu/2}^{\prime}(z) = \left( \frac{b}{2z} - \frac{1}{2}\right) \mathcal{M}_{\varkappa,\mu/2}(z) + \gamma \Phi^{\prime}(a,b,z) \\[1ex]
		&\begin{split}
			\mathcal{M}_{\varkappa,\mu/2}^{\prime\prime}(z) & = \left[ \left( \frac{b}{2z} - \frac{1}{2}\right)^2 - \frac{b}{2z^2}\right] \mathcal{M}_{\varkappa,\mu/2}(z) + \gamma \left(\frac{b}{z} - 1 \right) \Phi^{\prime}(a,b,z) \\[1ex] & + \gamma \Phi^{\prime\prime}(a,b,z)
	    \end{split}
	\end{split}
\end{equation}\\
\indent Using entries 13.4.7 - 13.4.9 in \cite{Abramowitz}, it can be shown that

\begin{equation}
	\label{eq:Phi dprime}
	\Phi^{\prime\prime}(a,b,z) = \left(1 - \frac{b}{z}\right)\Phi^{\prime}(a,b,z) + \frac{a}{z}\Phi(a,b,z)
\end{equation}\\
\noindent which simplifies $\mathcal{M}_{\varkappa,\mu/2}^{\prime\prime}(z)$ to

\begin{equation}
	\label{eq:simplified M dprime}
	\mathcal{M}_{\varkappa,\mu/2}^{\prime\prime}(z) = \left[\left(\frac{b}{2z} - \frac{1}{2}\right)^2 - \frac{b}{2z^2}\right] \mathcal{M}_{\varkappa,\mu/2}(z) + \frac{\gamma a}{z} \Phi(a,b,z)
\end{equation}\\
\noindent Return to (\ref{eq:simplified series}) and let $\mathcal{D}(z) = \mathcal{M}_{\varkappa,\mu/2}(z) \mathcal{M}_{\varkappa,\mu/2}^{\prime}(z) + z \mathcal{M}_{\varkappa,\mu/2}^{\prime\prime}(z) \mathcal{M}_{\varkappa,\mu/2}(z)$. After substituting $\mathcal{M}_{\varkappa,\mu/2}^{\prime}(z)$ and $\mathcal{M}_{\varkappa,\mu/2}^{\prime\prime}(z)$ from (\ref{eq:M prime and dprime}) and (\ref{eq:simplified M dprime}), respectively, and using the shorthand notation $\Phi = \Phi(a,b,z)$ (with a similar interpretation for $\Phi^{\prime}$ and $\Phi^{\prime\prime}$), $\mathcal{D}$ becomes

\begin{equation}
	\label{eq:D(z)}
	\begin{split}
		\mathcal{D}(z) & = \gamma \left(\frac{b}{2z} - \frac{1}{2}\right) \Phi \mathcal{M}_{\varkappa,\mu/2}(z) + \gamma^2 \Phi \Phi^{\prime} \\[1ex] & + z \gamma \Phi \left\{ \left[\left(\frac{b}{2z} - \frac{1}{2}\right)^2 - \frac{b}{2z^2}\right] \mathcal{M}_{\varkappa,\mu/2}(z) + \frac{\gamma a}{z} \Phi\right\}
	\end{split}
\end{equation}\\
\indent Now write (\ref{eq:D(z)}) entirely in terms of the Whittaker $\mathcal{M}$ function

\begin{equation}
	\label{eq:D in terms of M}
	\begin{split}
		\mathcal{D}(z) & = z \mathcal{M}^2_{\varkappa,\mu/2}(z) \left[\left( \frac{b}{2z} - \frac{1}{2}\right)^2 - \frac{b}{2z^2} + \frac{a}{z} - \frac{1}{2z} + \frac{b}{2z^2}\right]\\[1ex] & + \mathcal{M}_{\varkappa,\mu/2}(z) \left[ \mathcal{M}_{\varkappa,\mu/2}^{\prime}(z) - \left(\frac{b}{2z} - \frac{1}{2}\right)\mathcal{M}_{\varkappa,\mu/2}(z)\right]
	\end{split}
\end{equation}\\
\noindent Recognizing that $\mathcal{M}_{\varkappa,\mu/2}(z) \mathcal{M}_{\varkappa,\mu/2}^{\prime}(z) = \frac{1}{2} \frac{d}{dz} \mathcal{M}_{\varkappa,\mu/2}^2(z)$ and noting from (\ref{eq:parameter defs}) that $b/2-a = \varkappa$ and $b^2 - 2b = 1 + \mu^2$, (\ref{eq:D in terms of M}) can also be written as

\begin{equation}
	\label{eq:simplified D(z)}
	\mathcal{D}(z) = z \mathcal{M}_{\varkappa,\mu/2}^2(z) \left[ \frac{\mu^2-1}{4z^2} - \frac{\varkappa}{z} + \frac{1}{4}\right] + \frac{1}{2} \frac{d}{dz} \mathcal{M}_{\varkappa,\mu/2}^2(z)
\end{equation}\\
\noindent With $z = ct$ in (\ref{eq:simplified D(z)}), substituting back into (\ref{eq:simplified series}) yields the final expression

\begin{equation}
	\label{eq:final expression}
	\begin{split}
		I & = \int \left(\frac{1}{2} - \frac{\varkappa}{ct}\right) \mathcal{M}_{\varkappa,\mu/2}^2(ct)d(ct) = -ct [\mathcal{M}_{\varkappa,\mu/2}^{\prime}(ct)]^2 + \mathcal{D}(ct) \\[1ex] & = -(ct) \mathcal{M}_{\varkappa,\mu/2}^2(ct) \left( -\frac{1}{4} + \frac{\varkappa}{ct} + \frac{1-\mu^2}{4c^2 t^2}  \right) - ct [\mathcal{M}_{\varkappa,\mu/2}^{\prime}(ct)]^2 \\[1ex] & \quad\quad\quad + \frac{1}{2} \frac{d}{d(ct)} \mathcal{M}_{\varkappa,\mu/2}^2(ct)
	\end{split}
\end{equation}

\end{appendices}

\bibliography{sn-bibliography}

\end{document}